\documentclass[11pt]{amsart}

\usepackage{geometry}                
\usepackage{amsthm}
\usepackage{graphicx}
\usepackage{color}
\usepackage{amssymb}
\usepackage{paralist}
\usepackage{mathrsfs}
\usepackage{caption}
\usepackage{subcaption}

\usepackage[colorlinks,pdftex]{hyperref} 
	\hypersetup{
	    bookmarks=true,         
	    unicode=false,          
	    pdftoolbar=true,        
	    pdfmenubar=true,        
	    pdffitwindow=false,     
	    pdfstartview={FitH},    
	    pdftitle={A Polyhedral Proof of the Matrix Tree Theorem},    
	    pdfauthor={Aaron Dall and Julian Pfeifle},     
	    pdfnewwindow=true,      
	    colorlinks=true,       
	    linkcolor=red,          
	    citecolor=blue,        
	    filecolor=magenta,      
	    urlcolor=blue,           
	    citebordercolor={0 1 1},
	    linkbordercolor={0 0 1}
		}

\makeatletter
\newtheorem*{rep@theorem}{\rep@title}
\newcommand{\newreptheorem}[2]{%
\newenvironment{rep#1}[1]{%
 \def\rep@title{#2 \ref{##1}}%
 \begin{rep@theorem}}%
 {\end{rep@theorem}}}
\makeatother
\newtheorem{thm}{Theorem}
\newreptheorem{theorem}{Theorem}
\newtheorem{prop}[thm]{Proposition} 
 
\newtheorem{cor}[thm]{Corollary}
\newtheorem{con}[thm]{Conjecture}

\theoremstyle{remark}

\newtheorem{rmk}[thm]{Remark}
\newtheorem*{example*}{Running Example}
\newtheorem*{example2*}{Example}

\theoremstyle{definition}
 
\def\AA{\mathcal{A}}
\def\BB{\mathcal{B}}
\def\CC{\mathcal{C}}

\def\HH{\mathcal{H}}
\def\II{\mathcal{I}}

\def\LL{\mathcal{L}}
\def\MM{\mathcal{M}}

\def\F{\mathbb{F}}

\def\R{\mathbb{R}}
\def\Z{\mathbb{Z}}

\def\a{\mathbf{a}}

\def\u{\mathbf{u}}
\def\v{\mathbf{v}}
\def\w{\mathbf{w}}

\DeclareMathOperator{\aff}{aff}

\DeclareMathOperator{\conv}{conv}

\DeclareMathOperator{\im}{im}

\DeclareMathOperator{\posHull}{posHull}

\DeclareMathOperator{\rank}{rank}
\DeclareMathOperator{\Relint}{rel ~int}
\DeclareMathOperator{\sign}{sign}

\DeclareMathOperator{\Spec}{Spec}

\DeclareMathOperator{\vol}{vol}
\DeclareMathOperator{\diag}{diag}

\newcommand{\rL}{{\textsf{\upshape L}}}
\newcommand{\rM}{{\textsf{\upshape M}}}
\newcommand{\rN}{{\textsf{\upshape N}}}

\newcommand{\rZ}{{\textsf{\upshape Z}}}

\def\supp{\mathrm{supp}\:}

\newcommand{\cZono}[1]{\rZ_{0}(#1)}
\newcommand{\defn}[1]{\textbf{#1}}

\newcommand{\isomorphic}{\cong}

\newcommand{\vecone}{\boldsymbol{1}}
\newcommand{\veczero}{\boldsymbol{0}}

\graphicspath{{pictures/}}

\makeatletter
\def\@BracContents{} 
\newcommand{\@BracKern}{\kern-\nulldelimiterspace}
\newcommand{\suchthat}{%
   \nonscript\;
   \ifnum\currentgrouptype=16
     \middle|
   \else
     \@suchthat
   \fi
   \nonscript\;} 
\newcommand{\@suchthat}{%
   { 
   \let\suchthat\@empty 
   \left.\@BracKern 
   \vphantom{\@BracContents} 
   \middle| 
   \right.\@BracKern 
   } 
 }
\makeatother

\title{A Polyhedral Proof of the Matrix Tree Theorem}
\author{Aaron Dall}
\author{Julian Pfeifle}

\thanks{Both authors were partially supported by the project MINECO MTM2012-30951/FEDER. The first author received additional support from the MCINN grant BES-2010-030080. The second author received additional support from grants EUI-EURC-2011-4306, MTM~2011-24097 and 2009-SGR-1040.}

\address{Departament de Matem\`atica Aplicada II, 
Universitat Polit\`ecnica de Catalunya, 
Jordi Girona 1-3,
E-08034 Barcelona}

\email{[aaron.dall,julian.pfeifle]@upc.edu}	

\keywords{Regular matroid, unimodular matrix, cocircuit lattice, zonotope}		

\subjclass[2010]{Primary 52B40; Secondary 52C40.}	

\begin{document}
\maketitle

\begin{abstract}
The classical matrix tree theorem relates the number of spanning trees of a connected graph with the product of the nonzero eigenvalues of its Laplacian matrix.	
The class of regular matroids generalizes that of graphical matroids, and a generalization of the matrix tree theorem holds for this wider class.

We give a new, geometric proof of this fact by showing via a dissect-and-rearrange argument that two combinatorially distinct zonotopes associated to a regular matroid have the same volume.
Along the way we prove that for a regular oriented matroid represented by a unimodular matrix, the lattice spanned by its cocircuits coincides with the lattice spanned by the rows of the representation matrix. 

Finally, by extending our setup to the weighted case we give new proofs of recent results of~An et~al.\ on weighted graphs, and extend them to cover regular matroids. 

No use is made of the Cauchy-Binet Theorem nor divisor theory on graphs.
\end{abstract}

\section*{Introduction}
\label{section:intro}

The matrix tree theorem is a classical result in algebraic graph theory that relates the number of spanning trees of a connected graph $G$ with the product of the nonzero eigenvalues of the Laplacian matrix of $G$.
	\begin{thm}[Kirchoff \cite{kirchhoff1847ueber}]
	\label{thm:MatrixTree}
		Let $G$ be a connected graph on $n$ vertices with $s$ spanning trees and whose Laplacian $\rL$ has nonzero eigenvalues $\lambda_{1}, \dots, \lambda_{n-1}$. Then $$\prod_{i\in [n-1]} \lambda_{i} = ns.$$
	\end{thm}
A classical proof of this theorem proceeds as follows (see \cite{aigner2000proofs}, \cite{brouwer2011spectra}).
First one shows that every principal minor of $\rL$ is equal to the sum of the squares of the maximal minors of the signed vertex-edge incidence matrix of $G$.
Then one computes that such a minor is equal to $\pm1$ if the edges of $G$ corresponding to the columns of the submatrix span a tree.
Finally, one verifies that the maximal principal minor is precisely $1/n$ times the product of the nonzero eigenvalues of~$\rL$ using the characteristic polynomial.

Many generalizations of the theorem exist including for weighted graphs, simplicial complexes, and regular matroids.
In the last case, the regular matroid matrix tree theorem is the following particular case of Theorem 3 in \cite{maurer1976matrix}. 
	\begin{thm}
	\label{theorem:MatroidMTT}
		Let $\MM$ be a rank $d$ regular matroid represented by a $d \times n$ unimodular matrix $\rM$ of full rank, and let $\rL = \rM \rM^{\top}$.
		Then the number of bases of $\MM$ is $\lambda_{1}\cdots \lambda_{d}$, where $\lambda_{1},\dots, \lambda_{d}$ are the eigenvalues of $\rL$ .
	\end{thm}

In this paper we recast this result into the domain of polyhedral geometry by considering the zonotopes generated by the columns of the matrices $\rM$ and $\rL$ and proving that the volumes of these two zonotopes are the same although their combinatorial structures are in general vastly different. 
	\begin{thm}
	\label{theorem:polytopalMatroidMTT}
		Let $\rM$ be a unimodular matrix of full rank, and $\rL=\rM\rM^\top$.
		Then the volume of the zonotope $\rZ(\rM)$, the volume of the zonotope $\rZ(\rL)$, and the product of the eigenvalues of $\rL$ are all equal.	
	\end{thm}
When $\rM$ has full row rank, then so does $\rL$ and so the zonotope $\rZ(\rL)$ is a $d$-dimensional parallelepiped and $\rL$ has $d$ real nonzero eigenvalues. 
It follows that the volume of $\rZ(\rL)$ is exactly the determinant of $\rL$, which in turn is the product of the eigenvalues of $\rL$.
This shows that the last two quantities in Theorem \ref{theorem:polytopalMatroidMTT} are equal, and so the crucial part of the proof is to show that the zonotopes have the same volume.

After some preliminary results are presented in Section \ref{section:prelims}, we will prove Theorem \ref{theorem:polytopalMatroidMTT} in Section \ref{section:matroidProof} via a novel dissect-and-rearrange argument. 
In Section \ref{section:weightedCase} we generalize the previous result to the weighted case.
Finally, in Section \ref{section:graphicCase} we give a new polyhedral proof of the classical Matrix Tree Theorem that, while similar to the general proof for full rank matrices in the previous section, copes with the fact that the defining matrices $\rM$ and $\rL$ do not have full rank. The classical proof of the matrix tree theorem involves matrix calculations that rely on the total unimodularity of the signed vertex-edge incidence matrix of a graph $G$, i.e., that one has a totally unimodular representation of the matroid $\MM(G)$. Our polyhedral approach works even when the representation of $\MM(G)$ is only unimodular.



\section{Preliminaries}
\label{section:prelims}
	\subsection{Matrices}
	\label{subsection:matrices}

		Let $M$ be an $m \times n$ matrix.
		For $i \in [m]$ and $j \in [n]$ we write $m_{ij}$ for the entry in the $i^{\text{th}}$ row and $j^{\text{th}}$ column of $M$ and $M_{j}$ for the $j^{\text{th}}$ column of $M$.
		Throughout this paper $\rM$  will denote a $d \times n$ integer matrix of rank $r$.
		The matrix $\rM$ is \defn{unimodular} if all of its maximal minors are in $\{-1,0,1\}$ and is \defn{totally unimodular} if all of its minors are in $\{-1,0,1\}$.
		Note that a unimodular matrix $\rM$ remains unimodular after appending either a column of zeros or a copy of the column $\rM_{j}$.
	 	We define $\rL = \rL(\rM)$ to be the $d\times d$ symmetric matrix given by $\rL := \rM \rM^{\top}$ (where $\rM^{\top}$ denotes the transpose of $\rM$).
	 	As $\rL$ is symmetric, its eigenvalues $\lambda_{1}, \dots, \lambda_{d}$ are all real. 		

	 	For any ring $R$ we denote the set of all $R$-combinations of the columns of $M$ by ${_R}{\left \langle M \right \rangle}$.
	 	Note that when $R=\Z$ and $\rM$ is a unimodular matrix we have ${_{\Z}}{\left \langle \rM \right \rangle}= \Z^{m} \cap \im \rM$.
		We note here that in this paper we reserve the term lattice for a free discrete subgroup of a vector space.
	\subsection{Graphs}		
	\label{Graphs}

		A graph $G = (V,E)$ consists of a finite set of vertices $V$ together with a finite multiset of edges $E$ consisting of subsets of $V$ of size two.
		We assume that both $V$ and $E$ are ordered, and if $V$ has $n$ elements we identify it with the set $[n]:= \{1,\dots,n\}$ with the usual ordering.
		A \defn{loop} in $G$ is an edge of the form $\{v,v\}$ and two edges are \defn{parallel} if they are equal (as sets).
		A graph is \defn{connected} if between any two vertices $v$ and $w$ there are edges $\{v,u_{1}\},\{u_{1},u_{2}\},\dots, \{u_{i},w\}$ in $E$.
		A \defn{weighted graph} is a graph $G$ together with an assignment of a weight (a real number in this paper) to each edge.
		An \defn{orientation} of $G$ assigns a direction to each edge, i.e., assigns to each edge $\{v,w\}$ one of two ordered pairs $(v,w)$ or $(w,v)$, where the first (respectively, second) element is called the \defn{tail} (respectively, \defn{head}) of the edge.
		A graph with an orientation is called a \defn{directed graph}, or digraph for short.

		The \defn{vertex-edge incidence matrix} $\rN(G)$ of an (unweighted) digraph $G$ is the matrix with one row for each vertex, one column for each edge, and where the entry $\rN_{v,e}$ is $1$ (respectively, $-1$) if $v$ is the head (resp. tail) of $e$ and $0$ otherwise.
		The \defn{Laplacian} of $G$ is then the matrix $\rL = \rN\rN^{\top}$.
		If $G$ is a weighted digraph with edge weights $\omega = (w_{1}, \dots, w_{m})$ then the \defn{weighted incidence matrix} of $G$ is $\rN_{w}(G) := \rN(G)D$, where $D$ is the $m \times m$ diagonal matrix with $D_{ii} = w_{i}$.
		The \defn{weighted Laplacian} is then $\rL_{w} := \rN_{w}\rN^{\top} = \rN D \rN^{\top}$.

	\subsection{Matroids}		
	\label{subsection:Matroids}

		We now give the pertinent definitions and facts about matroids and oriented matroids, essentially following \cite{oxley1992matroid} and \cite{bjorner1999oriented}, respectively. 
		A \defn{matroid} $\MM = (E, \II)$ is an ordered pair consisting of a ground set $E$ and a collection $\II$ of subsets of $E$ that satisfy the following \defn{independent set axioms}:
			\begin{enumerate}[\bfseries{I}1] 
				\item $\emptyset \in \II$;
				\item $\II$ is closed with respect to taking subsets; and
				\item if $I_{1},I_{2} \in \II$ with $|I_{1}|\le |I_{2}|$, then there is some $e \in I_{2} \setminus I_{1}$ such that $I_{1}\cup\{e\} \in \II$. 
			\end{enumerate}
		The \defn{bases} $\BB$ of a matroid $\MM$ is the subset of $\II$ consisting of independent sets of maximal size.
		Clearly, the sets $\BB$ and $\II$ of a matroid determine each other.

		Let $M$ be an $m \times n$ matrix with entries in a field $\mathbb{F}$. 
		Then the archetypal example of an independent set matroid is $\MM = ([n], \II(M))$ where $[n]:=\{1,\dots, n\}$ is an indexing set for the columns of $M$ and $\II(M)$ consists of all subsets of (indices of) columns of $M$ that are linearly independent in the $m$-dimensional vector space $V(m,\mathbb{F})$ over $\mathbb{F}$.
		In this case we write $\MM = \MM(M)$.
		A matroid $\MM$ is called $\mathbb{F}$-\defn{representable} if there exists a matrix~$M$ with entries in $\mathbb{F}$ such that $\MM = \MM(M)$.
		It is immediate that $B \in \BB$ is a basis of an $\mathbb{F}$-representable matroid $\MM(M)$ if and only if the collection $\{M_{i} \suchthat i \in B\}$ is a basis for $V(m,\mathbb{F})$.

		Let $\MM=([n],\II)$ be an arbitrary matroid. 
		A subset $C \subset [n]$ is a \defn{circuit} of $\MM$ if $C$ is a minimal dependent set.
		The set of circuits of $\MM$ is denoted $\CC = \CC(\MM)$.
		A matroid $\MM$ is \defn{connected} if for every pair of elements $e\neq f$ in $E$ there is a circuit $C \in \CC$ containing both. 
		A \defn{loop} of $\MM$ is a singleton that is also a circuit.
		Two elements $f,g \in \MM$ are \defn{parallel} if $\{f,g\}$ is a circuit.

		Given a matroid $\MM=([n],\II)$ with bases $\BB$, let $\BB^{*} = \{E - B \suchthat B \in \BB\}$ and let $\II^{*}$ be the collection of all subsets of elements of $\BB^{*}$.
		Then the matroid $\MM^{*} :=([n],\II^{*})$ is the \defn{dual} matroid of $\MM$.
		The sets $\BB^{*},\II^{*},\CC^{*}$ of bases, independent sets, and circuits of the dual $\MM^{*}$ are called, respectively, the \defn{cobases}, \defn{coindependent sets} and \defn{cocircuits} of~$\MM$.		

		Let $\CC = \{C_{1}, \dots, C_{m}\}$ be the circuits of $\MM$.
		Then the \defn{circuit incidence matrix} of $\MM$ is the $m \times n$ matrix $A(\CC)$ with entry $a_{ij}$ equal to 1 if $j$ is in $C_{i}$ and equal to 0 otherwise.
		The cocircuit incidence matrix $A(\CC^{*})$ is defined analogously.
		The matroid~$\MM$ is \defn{orientable} if one replace some of the nonzero entries of $A'(\CC)$ and $A'(\CC^{*})$ by~$-1$ such that, if a circuit $C$ and cocircuit $C^{*}$ have nonempty intersection, then both of the sets $\{i \in [n] \suchthat \a_{C}(i) = \a_{C^{*}}(i) \neq 0\}$ and $\{i \in [n] \suchthat \a_{C}(i) = - \a_{C^{*}}(i) \neq 0\}$ are nonempty.

		In this paper we will work with the class of \defn{regular} matroids characterized in the following theorem. (See Lemma 2.2.21, Theorem 6.6.3, and Corollary 13.4.6 in \cite{oxley1992matroid}.)
		\begin{thm}
		\label{theorem:regularMatroidCharacterization}
			For a matroid $\MM$ the following are equivalent:
			\begin{enumerate}[\upshape(1)]
				\item $\MM$ is regular;
				\item $\MM$ is $\F_{2}$-representable and orientable;
				\item $\MM$ is representable over $\R$ by a unimodular matrix;
				\item $\MM$ is representable over $\R$ by a totally unimodular matrix; and
				\item the dual of $\MM$ is regular.			
			\end{enumerate}
			Moreover, if $\MM$ is regular,  $\rM$ is a totally unimodular matrix that represents $\MM$ over $\R$ and $\mathbb{F}$ is any other field, then $\rM$  is an $\mathbb{F}$-representation of $\MM$ when viewed as a matrix over $\mathbb{F}$.
		\end{thm}
		As noted in Subsection \ref{subsection:matrices}, a unimodular matrix $\rM$ remains unimodular after adding either a column of zeros or a copy of the column $M_{j}$.
		In matroid terminology, adding a column of zeros corresponds to adding a loop to the matroid $\MM(\rM)$ while adding a copy of a column gives a parallel element.

		If $\MM$ is an $\R$-representable matroid on $n$ elements with representation $M$ and if $D$ is any $n \times n$ diagonal matrix with nonzero real entries on the diagonal, then the matroids $\MM = \MM(M)$ and $\MM(MD)$ are isomorphic.
		By way of analogy with the graphical case, we write $M_{w}$ for $MD$, where $w_{i} = D_{ii}$, and call $M_{w}$ a \defn{weighted representation} of $\MM$.
		We also define the \defn{weighted Laplacian} of $\MM(M)$ with respect to $D$ to be~$L_{w} := MDM^{\top}$.

	\subsection{Oriented Matroids}
	\label{subsection:orientedMatroids}
		We now turn to oriented matroids that are representable over $\R$.	
		Let $\AA = \{\a_{1}, \dots, \a_{n}\} \subset \R^{d}$ be a vector configuration that spans the vector space $\R^{d}$.
		Then the \defn{covectors} of the oriented matroid $\MM(\AA)$ are the elements of the set 
			\begin{align*}
				\mathcal V^{*} :&= \{(\sign f(\a_{1}), \dots, \sign f(\a_{n})) \suchthat f: \R^{d} \to \R \text{ linear functional}\}\\
				&\subseteq \{-1,0,1\}^{n}.
			\end{align*}
		The \defn{cocircuits} of the oriented matroid $\MM(\AA)$ are the minimal elements of the poset $(\mathcal{V}^{*}, \prec)$, where the relation $\prec$ is defined by extending $0 \prec \pm 1$ component-wise.
		In the next subsection, we recall how this poset is related to the face poset of the zonotope generated by the matrix $M$ whose columns are the $\a_{i}$.
		Here we conclude by recalling how to retrieve the covectors of the oriented matroid $\MM(\AA)$ from a certain subspace arrangement in $\R^{n}$ and how to obtain the underlying unoriented matroid $\underline{\MM}(\AA)$.

		Let $M$ be the matrix whose columns are $\a_{1}, \dots, \a_{n}$.
		Then the covectors of $\MM(\AA)$ can be read off from the hyperplane arrangement induced by the coordinate hyperplanes of~$\R^{n}$ in the rowspace of $M$. 
		To see this first consider the $n$ columns of $M$ as elements of the dual vector space $(\R^{d})^{*}$.
		Then each $\a_{i}$ defines a hyperplane in $\R^{d}$ given by $\mathcal{H}_{i} := \{x \in \R^{d} \suchthat \left \langle \a_{i}, x \right \rangle = 0\}$ for $i \in [n]$.
		Defining $\mathcal{H}_{i}^{+} := \{x \in \R^{d} \suchthat \left \langle \a_{i}, x \right \rangle > 0\}$ and $\mathcal{H}_{i}^{-} = \R^{d}\setminus(\HH_{i} \cup \HH_{i}^{+})$, assign to each $x \in \R^{d}$ a sign vector $\sigma(x) \in \{-1,0,1\}^{n}$ whose $i^{\text{th}}$~coordinate is $1$ (respectively, $0,-1$) if $x$ is in $\mathcal{H}_{i}^{+}$ (respectively, $\mathcal{H}_{i}, \HH_{i}^{-}$).
		The set of all points in $\R^{d}$ that receive the same sign vector $\sigma$ form a relatively open topological cell (which we label with $\sigma$) and the union of all such cells is $\R^{d}$.
		The sign vectors that occur are precisely the covectors of the oriented matroid $\MM(\AA)$, and the sign vectors that label $1$-dimensional cells are the cocircuits.

		Now consider the subspace arrangement in the rowspace of $M$ induced by the coordinate hyperplane arrangement in $\R^{n}$ (oriented in the natural way), which we denote by $\mathscr{H}(M)$. 
		A point $y$ in the rowspace of $M$ satisfies $y_{i} = 0$ (respectively, $y_{i}> 0, y_{i}< 0$) if and only if any point $x \in \R^{d}$ with $y = M^{\top}\!x$ lies on the hyperplane $\HH_{i}$ (respectively, in $\HH_{i}^{+}, \HH_{i}^{-}$) as defined in the previous paragraph.
		So the oriented matroid coming from the hyperplane arrangement in the rowspace of $M$ induced by the coordinate hyperplanes in $\R^{n}$ is exactly $\MM(\AA)$.
		
		The preceding discussion tells us that the rowspace of $M$ intersects exactly those cells of the coordinate hyperplane arrangement labeled by the covectors of $\MM(M)$.
		It should be noted, however, that in general the covectors themselves do not lie in the rowspace of $M$ even when $M$ is a totally unimodular matrix (see the upcoming Remark~\ref{rem:covectors}).
		We show in Theorem \ref{thm:cocircuitSpace} that, when $\rM$ is a unimodular matrix, every cocircuit of the oriented matroid $\MM(\rM)$ does lie in the rowspace of $\rM$, and in fact, the set of cocircuits is a spanning set for the lattice ${_{\Z}}\!\left \langle \rM^{\top} \right \rangle$. 

		Finally, to each oriented matroid $\MM$ one associates the underlying unoriented matroid~$\underline{\MM}$ whose cocircuits are obtained from the cocircuits of $\MM$ by forgetting signs, i.e., if~$C^{*}$ is a cocircuit of $\MM$, then $\underline{C}^{*}$ is a cocircuit of $\underline{\MM}(M)$ where $(\underline{C}^{*})_{i} = |(C^{*})_{i}|$.
		An oriented matroid is \defn{regular} if its underlying unoriented matroid is.
		Many statistics of an orientable matroid (e.g., the number of bases or the number of independent sets) remain invariant after orientation, and so when discussing these properties with respect to a given matroid $\MM(M)$, we often disregard the difference between the oriented matroid and the underlying unoriented matroid when no confusion can arise.

	\subsection{Zonotopes}

		A \defn{zonotope} is a Minkowski sum of a finite number of line segments.
		For an $m \times n$ matrix $M$ the \defn{zonotope} generated by $M$, denoted $\rZ(M)$, is the Minkowski sum of the line segments $\conv\{\veczero, M_{i}\}$.
		If $M$ has rank $d$ then the zonotope $\rZ(M)$ is a $d$-dimensional convex polytope that is centrally symmetric about its barycenter.
		We write $\rZ_{0}(M)$ for the translated copy of $\rZ(M)$ whose barycenter is at the origin, i.e., the Minkowski sum  $\sum_{i\in[n]} S_i$ where $S_i = \conv\{-\frac{1}{2}\rM_{i}, \frac{1}{2}\rM_{i}\}$. 
		
		A parallelepiped is \defn{half open} if it is it the Minkowski sum of half-open line segments. The next result, due to Stanley, gives a decomposition of a zonotope into half-open parallelepipeds of various dimensions. 
		
		\begin{thm}[\cite{stanley1991zonotope} Lemma 2.1]
		\label{theorem:zonotopeDecomp}
			Let $M$ be a rank $d$ matrix and $\II$ be the independent sets of the matroid $\MM(M)$.
			Then the zonotope $\rZ(M)$ is the disjoint union of half-open parallelepipeds 
			\[\Pi_{I} := \left\{\sum_{i \in I}\alpha_{i}\widetilde{M_{i}} \suchthat \alpha_{i} \in [0,1)\right \},\]
			where $\widetilde{M_{i}}$ is either $M_{i}$ or $-M_{i}$.
		\end{thm}
		
		The parallelepipeds of maximal dimension in the above theorem are generated by maximal independent subsets of the columns of $M$.
		As their union covers $\rZ(M)$ up to a set of measure zero, it follows that the volume of the zonotope is the sum of the volumes of the parallelepipeds generated by the bases of $\MM(M)$.
		For a fixed basis $B \in \BB(M)$, the volume of $\Pi_{B}$ is simply the absolute value of the determinant of (the matrix whose columns are elements of) $B$.
		When $M$ is unimodular each of these determinants is
		$\pm 1$ and so we have the following corollary:
		\begin{cor}
		\label{corollary:volumeOfTUzonotope}
		The volume of a zonotope generated by a  unimodular matrix $\rM$ is equal to the number of bases in the  regular matroid $\MM(\rM)$, i.e., $\vol (\rZ(\rM)) = |\BB(\rM)|$.	
		\end{cor}

		A zonotope $\rZ(M) \subset \R^{d}$ generated by a representation $M$ of a regular matroid $\MM$ can be used to tile its affine span.
		More precisely, a polytope $P \subset \R^{d}$ is said to \defn{tile} its affine span $S$ if there is a polyhedral subdivision of $S$  whose maximal cells are translates of~$P$.
		The next theorem, due to Shepard~\cite{shephard1974space}, tells us that a zonotope tiles its affine span exactly when the underlying matroid is regular.
		\begin{thm}
			\label{theorem:spaceTiling}
			A zonotope $\rZ(M)$ tiles its affine span if and only if the matroid $\MM(M)$ is regular.
		\end{thm}
		Note that in the above theorem $M$ is not required to be unimodular but only a representation over $\R$ of a regular matroid.
		This distinction will become important later on when we discuss the space-tiling properties of the zonotope generated by the Laplacian of a connected graph which, though not itself a unimodular matrix, is nevertheless a representation of the regular matroid $\MM(\rN^{\top})$, where $\rN$~is the signed vertex-edge incidence matrix of the graph.

		Any $k$-dimensional zonotope $\rZ \subseteq \R^{n}$ can be viewed as the projection of the unit $n$-cube.
		Moreover, in \cite{mcmullen1984volumes} one finds the following theorem in which the (Euclidean) $d$-dimensional volume of $\rZ$ and the $n-d$-dimensional volume of a certain zonotope $\overline{\rZ}$ in the orthogonal complement of the linear hull of $\rZ$ are shown to be the same.
		\begin{thm}
		\label{theorem:associatedVolumes}
			If $\rZ$ and $\overline{\rZ}$ are images of the unit cube in $\R^{n}$ under orthogonal projection onto orthogonal subspaces of dimension $d$ and $n-d$, respectively, and $\vol_{k}$ denotes the $k$-dimensional Euclidean volume form, then $\vol_{d}(\rZ) = \vol_{n-d}(\overline{\rZ})$.
		\end{thm}

		To conclude this subsection we fulfill our promise from the previous one and give the relationship between covectors of an oriented matroid $\MM(M)$ (for an arbitrary matrix~$M$) and the faces of the zonotope $\cZono{M}$.
		Let $P:=(\mathcal{V}^{*},\prec)$ be the poset of covectors of $\MM(M)$ where $\prec$ is the component-wise extension of $0\prec \pm 1$, and let $\mathcal{F}$ be the poset whose elements are the faces of $\cZono{M}$ ordered by inclusion.
		Then $P$ is anti-isomorphic to $\mathcal{F}$ as witnessed by the order-reversing bijection that sends a covector $\v = (v_{1}, \dots, v_{n})$ to the face
			\[
		   F_{\v} 
		   = \sum_{i: v_{i} = 1} \tfrac{1}{2}M_{i} 
		   - \sum_{i: v_{i} =-1} \tfrac{1}{2}M_{i} 
		   + \sum_{i: v_{i} = 0} S_{i},
		        \]
		where $S_{i} = \conv\left\{-\frac{1}{2}M_{i}, \frac{1}{2}M_{i}\right\}$.
		Note that the facets of $\cZono{M}$ correspond to the cocircuits of the oriented matroid $\MM(M)$.
		 Now consider the barycenters $\pm \beta_{1}, \dots, \pm \beta_{r}$ of the facets $\pm F_{1}, \dots, \pm F_{r}$ of $\cZono{M}$.
		If $\CC^{*}_{i}$ is the cocircuit corresponding to the facet $F_{i}$, then it is clear from the above expression that $\beta_{i} = \frac{1}{2}M\CC^{*}_{i}$. 
		For the formulation of Corollary~\ref{corollary:taleOfTwoLattices} below, it turns out to be more appropriate to work with the \defn{scaled barycenter matrix} $B = B(M)$ whose columns are the~$\beta_i$ scaled by a factor of~$2$.

 \subsection{Lattices}

		Before turning to the proof of the main result in the next section, we review the necessary terminology and results from lattice theory.
		Our notation follows~\cite{vallentin2004note} and all proofs can be found either there or in~\cite{tutte1971matroids}.

		Let $\LL \subset \R^{n}$ be a lattice, i.e., a free discrete subgroup of $\R^{n}$.
		Then $\LL$ gives rise to the oriented matroid $\MM(\LL)$ whose covectors are $\mathcal{V^{*}} = \{\sign(\v) \suchthat \v \in \LL\}.$
		The \defn{support} of a vector $\v \in \LL$ is the set $\supp(\v) = \{i \in [n] \suchthat \v_{i} \neq 0\}$.
		A nonzero vector $\v \in \LL$ is \defn{elementary} if its coordinates lie in $\{-1,0,1\}$ and it has minimal support in $\LL\setminus{\veczero}$.
		Two vectors in $\LL$ are \defn{conformal} if their component-wise product is in $\R^{n}_{\ge 0}$. 

		A \defn{zonotopal lattice} is a pair $(\LL,\left \langle \cdot, \cdot \right \rangle)$ where $\LL \subset \Z^{n}$ is a lattice, $\left \langle \cdot, \cdot \right \rangle$ is an inner product on $\R^{n}$ such that the canonical basis vectors are pairwise orthogonal, and for every $\v \in \LL \setminus \{\veczero\}$ there is an elementary vector $\u \in \LL$ such that $\supp(\u) \subseteq \supp(\v)$.
		The next proposition (Lemma 3.2 in \cite{vallentin2004note}) tells us that zonotopal lattices are generated by the cocircuits of their oriented matroids in an especially nice way.
		\begin{prop}
		\label{prop:zonotopalLattices}
			The elementary vectors of a zonotopal lattice $\LL$ are exactly the cocircuits of the oriented matroid $\MM(\LL)$.
			Moreover, every vector $\v \in \LL$ is the sum of pairwise conformal elementary vectors, and if the support of $\v$ equals the support of some elementary vector $\u$, then $\v$ is a scalar multiple of $\u$.
		\end{prop}

		As noted in Remark 4.2 of \cite{vallentin2004note}, the oriented matroid of a zonotopal lattice is regular. 
		Historically this was taken as the definition of a regular matroid (see Section 1.2 of \cite{tutte1971matroids}).		
		We reestablish this connection and give a modern proof for the fact that, for a regular oriented matroid $\MM(\rM)$ with cocircuits~$\mathcal{C}^{*}$ and $\rM$ unimodular, the lattices generated by $\rM^{\top}$ and $\mathcal{C}^{*}$ coincide (see Theorem \ref{thm:cocircuitSpace}).



\section{Proof of the Polyhedral Matroid Matrix Tree Theorem}
\label{section:matroidProof}

		Let $\MM$ be a regular rank $d$ matroid.
		If $\rM$ is a unimodular representation of $\MM$, then by Corollary \ref{corollary:volumeOfTUzonotope} the volume of the zonotope generated by $\rM$ is equal to the number of bases of $\MM$, $\vol(\rZ(\rM)) = |\BB(\MM)|$.
		When $\rM$ has full-row rank then so does the square matrix~$\rL$, and so the zonotope $\rZ(\rL)$ is a parallelepiped with volume $\det (\rL) = \lambda_{1}\cdots\lambda_{d}$, where the $\lambda_{i}$ are the eigenvalues of $\rL$.
		Using row operations that preserve unimodularity and then deleting any rows of zeros, any unimodular represention $\rM$ of $\MM$ can be transformed into a full row-rank unimodular representation of $\MM$, so without loss of generality we may assume $\rM$ is a full rank unimodular representation of $\MM$. 
		The proof of Theorem \ref{theorem:polytopalMatroidMTT} will be complete once we show that the zonotopes $\rZ(\rM)$ and $\rZ(\rL)$ have the same volume.

		\begin{rmk}
		When $\rM$ has nontrivial corank (as is the case, for example, when $\rM = \rN(G)$ is the signed incidence matrix of a graph), the zonotope $\rZ(\rL)$ is no longer a parallelepiped.
		This means that some care must be taken when showing that the volume of $\rZ(\rL)$ is the product of its nonzero eigenvalues.
		We sweep this detail under the rug in this section for ease of exposition, dealing with it in detail in the next section	where we use our techniques to prove the graphical matrix tree theorem.
		\end{rmk}

		Our first goal is to see that, when $\rM$ is a unimodular representation of a regular matroid, the lattices generated by $\rL$ and the scaled barycenter matrix $B$ coincide.
		(Note that we do not require $\rM$ to have full rank nor to be totally unimodular.)
		This fact is an immediate corollary of the following theorem.  
				\begin{thm}
				\label{thm:cocircuitSpace}
					Let $\MM$ be a regular oriented matroid on $n$ elements and $\rM$ be a unimodular matrix representing $\MM$ over $\R$. Then the lattices ${_{\Z}}\!\left \langle \rM^{\top} \right \rangle$ and ${_{\Z}}\!\left \langle \CC^* \right \rangle$, generated by the columns of~$\rM^{\top}$ and by the cocircuits of $\MM$, respectively, coincide.
				\end{thm}
				\begin{proof}			

					Recall that the subspace arrangement $\mathscr{H}=\mathscr{H}(\rM) \subset \R^{n}$ is obtained by intersecting the rowspace of $\rM$ with the coordinate hyperplane arrangement in $\R^{n}$.
					Clearly, the closure of any cell of $\mathscr{H}$ is the positive hull of the rays of $\mathscr{H}$ it contains and the sign vector of a cell is conformal to each of the rays contained in its closure.
					By the discussion in Section~\ref{subsection:Matroids},  the cocircuits of $\MM(\rM)$ are the sign vectors that label the rays of this arrangement. 
					Let $\rho$ be such a ray, labeled with the sign vector $\sigma$.
					We claim $\rho = \posHull(\sigma)$.

					Consider the polytope $\rho \cap [-1,1]^{n}$.
					The equations for the rowspace of $\rM$ are given by the kernel of~$\rM$ and it follows from Theorem \ref{theorem:regularMatroidCharacterization} that one can find a unimodular basis for $\ker \rM$ (see \cite{beck2014enumerating} Lemma 2.10 for details when $\rM$ is a full-rank totally unimodular matrix).
					 Thus the line segment $\rho \cap [-1,1]^{n}$ is the intersection of hyperplanes and halfspaces whose normal vectors can be viewed as the rows of a unimodular matrix.
					 Moreover, the equations and inequalities of the segment all have integer (in fact $\{0,\pm 1\}$) right-hand sides.
					 Thus, by Theorem 19.2 in \cite{schrijver1998theory}, we obtain that $\rho \cap [-1,1]^{n} = \conv\{\veczero, \v\}$ is a lattice segment with $\v \in \{-1,0,1\}^{n}$ .
					But then $\v = \sign(\v) = \sigma$, and so $\rho = \posHull(\sigma)$.
					In particular, every cocircuit $C^{*}$ of $\MM(\rM)$ is in the rowspace of $\rM$, and hence ${_{\Z}}\!\left \langle \CC^* \right \rangle \subseteq {_{\Z}}\!\left \langle \rM \right \rangle$.

					For the opposite inclusion, let $\w \in {_{\Z}}\!\left \langle \rM \right \rangle$ with sign vector $\sigma_{\w}$.
					Then, as the cell of~$\mathscr{H}$ labelled with $\sigma_{\w}$ is the positive hull of the rays it contains and the labels on these rays have minimal support, for any such ray $\rho$ we have $\supp(\rho) \subseteq \supp(\sigma_{\w})$.
					It follows immediately that ${_{\Z}}\!\left \langle \rM \right \rangle$ together with the standard inner product on $\R^{n}$ is a zonotopal lattice. 
					Moreover, the elementary vectors of ${_{\Z}}\!\left \langle \rM \right \rangle$ are those $\{-1,0,1\}$-vectors in the rowspace of $\rM$ that have minimal support, i.e., lie on a ray of the arrangement induced by the coordinate hyperplane arrangement.
					It follows that the elementary vectors of ${_{\Z}}\!\left \langle \rM \right \rangle$ are the cocircuits of $\MM(\rM)$ and, since elementary vectors of a zonotopal lattice span the lattice by Proposition~\ref{prop:zonotopalLattices}, the theorem follows. 
				\end{proof}

		                \begin{rmk}\label{rem:covectors}
		                  As we already mentioned, the covectors of $\MM$ do not always lie in the rowspace of $\rM$. Consider the totally unimodular matrix 
		\[
		\rN_{K_3}
		\ = \
		\begin{pmatrix}
		  -1 & -1 &  0\\
		  1  &  0 & -1\\
		  0  &  1 &  1
		\end{pmatrix}.
		\]
		By Theorem~\ref{thm:cocircuitSpace}, the rowspace of $\rN_{K_3}$ has a basis of cocircuits, for example 
		\[
		\CC^*
		\ = \
		\begin{pmatrix}
		   1 & 1 & 0 \\
		   0 & 1 & 1
		\end{pmatrix}.
		\]
		The lattice point $(1,2,1)$ lies in the rowspace of $\CC^*$. However, taking signs yields the covector $(1,1,1)$ of $\MM$, which does not lie in the rowspace of~$\CC^*$.\qed
		                \end{rmk}

		Recall that for an arbitrary matrix $M$, the columns of the scaled barycenter matrix $B=B(M)$ are the barycenters $\beta_i=\frac{1}{2}M\CC^{*}_{i}$ of $\cZono{M}$, scaled by~$2$.

			\begin{cor}
			\label{corollary:taleOfTwoLattices}
				Let $\MM$ be a regular oriented matroid on $n$ elements and $\rM$ be a unimodular matrix representing $\MM$ over $\R$.
				Then the lattices generated by the columns of $\rL$ and the columns of~$B$ are equal.
			\end{cor} 
			\begin{proof}		
				Theorem \ref{thm:cocircuitSpace} tells us that the lattices generated respectively by $\rM^{\top}$ and $\CC^{*}$ coincide, and therefore so do their images $\rL=\rM\rM^\top$ and $B=\rM\CC^*$ under $\rM$.
			\end{proof}

		 We now use the fact that the columns of $\rL$ are a basis for $\R^{d}$ to define a subdivision of $\cZono{\rM}$.
		For each sign vector $\epsilon \in \{+,-\}^{d}$ we define the following objects:
			\begin{itemize}
				\item the simplicial cone $\sigma_{\epsilon} := \posHull\{\epsilon_{i}\rL_{i} \suchthat i \in [d]\}$ (see Figure \ref{fig:ZNtoZLsub1});
				\item the vector $v_{\epsilon^-} := \sum_{i:\,  \epsilon_{i} = -} \rL_{i}$;
				\item the polytope $P_{\epsilon} := \sigma_{\epsilon} \cap \cZono{\rM}$ (see Figure \ref{fig:ZNtoZLsub2});
				\item  the polytope $Q_{\epsilon} := P_{\epsilon} + v_{\epsilon^-}$ (see Figure \ref{fig:ZNtoZLsub3}).
			\end{itemize}

			\begin{figure}[htbp]
			\centering	
				\begin{subfigure}{.3\textwidth}
				  \centering
				  \includegraphics[width=.9\linewidth]{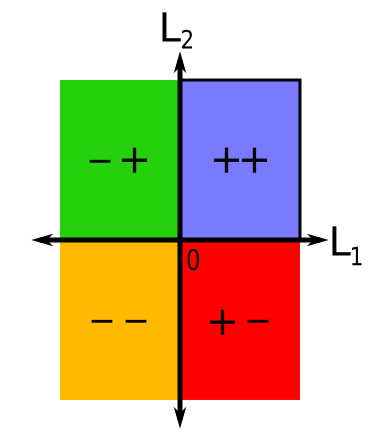}
				  \caption{The cones $\sigma_{\epsilon}$}
				  \label{fig:ZNtoZLsub1}
				\end{subfigure}%
				\begin{subfigure}{.3\textwidth}
				  \centering
				  \includegraphics[width=.9\linewidth]{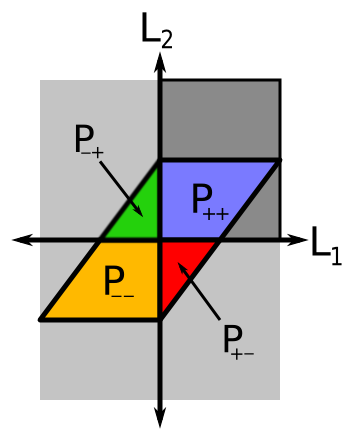}
				  \caption{$\cZono{\rM} = \bigcup P_{\epsilon}$}
				  \label{fig:ZNtoZLsub2}
				\end{subfigure}
				\begin{subfigure}{.3\textwidth}
				  \centering
				  \includegraphics[width=.9\linewidth]{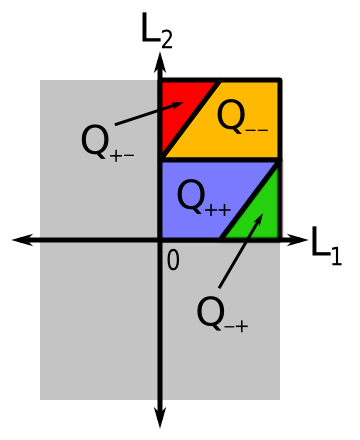}
				  \caption{$\rZ(\rL) = \bigcup Q_{\epsilon}$ }
				  \label{fig:ZNtoZLsub3}
				\end{subfigure}
						\caption{The polyhedra induced by sign vectors for the path on three vertices, after a change of coordinates that transforms the columns of $\rL$ to the standard basis.}
			\label{fig:ZNtoZL}
			\end{figure}

		\begin{example2*}
			Consider the path on three vertices with edges oriented so that $i \rightarrow j$ if $i<j$.
		A full-rank representation for the independent set matroid $\MM(\rN)$ of the signed incidence matrix of this graph is given by the matrix on the left below, while the corresponding Laplacian is the matrix on the right:
			\[\rM = \begin{pmatrix}
				-1 & 0 \\
				1 & -1 		
			\end{pmatrix}
			\hspace{44pt}
			\rL = \begin{pmatrix}
				1 & -1 \\
				-1 & 2 		
			\end{pmatrix}.\]
			Both of the zonotopes $\rZ(\rM)$ and $\rZ(\rL)$ are two dimensional parallelepipeds and Figure \ref{fig:ZNtoZL} illustrates the families $\sigma_{\epsilon}, P_{\epsilon}$, and $Q_{\epsilon}$ as $\epsilon$ varies over all sign vectors for this example after a suitable coordinate transformation. 
			Note that the zonotope of the Laplacian is the parallelepiped in the positive quadrant shaded dark grey. \qed
		\end{example2*}

		Clearly the union of the $P_{\epsilon}$ over all sign vectors is the zonotope $\cZono{\rM}$ and the intersection of any two of them is a face of both. 
		We now prove Theorem \ref{theorem:polytopalMatroidMTT} by showing that the union of the $Q_{\epsilon}$ is in fact $\rZ(\rL)$ and that any two $Q_\epsilon$ intersect in a set of measure~zero.

			\begin{reptheorem}{theorem:polytopalMatroidMTT}
				Let $\MM$ be a regular oriented matroid on $n$ elements, let $\rM$ be a unimodular matrix representing $\MM$ over $\R$, and put $\rL=\rM\rM^\top$.
				Then the volume of the zonotope $\rZ(\rL)$ equals the volume of the zonotope $\rZ(\rM)$.	
			\end{reptheorem}
			\begin{proof}
				By Corollary \ref{corollary:taleOfTwoLattices}, the line segment $\conv\{\veczero, \rL_{i}\}$ intersects some proper face of $\cZono{\rM}$ and the point of intersection is the barycenter of both.
				In particular, the distance between any two points of~$\cZono{\rM}$ in the direction parallel to $\rL_{i}$ is less than or equal to~$||\rL_{i}||$, with equality if and only if the points lie in opposite faces of $\cZono{\rM}$ intersected by the line ${}_{\R}\langle\rL_{i} \rangle$. 

				First we show that $\bigcup_{\epsilon}Q_{\epsilon} \subseteq \rZ(\rL)$.

				Let $H_{1} = \langle \rL_{i} \suchthat i \in \{2,\dots,d\} \rangle_{\R}$ be the hyperplane spanned by all columns of $\rL$ except for $\rL_{1}$ and let $H^{+}_{1}$ be the open halfspace bounded by $H_{1}$ and containing $\rL_{1}$. 
				For  $p \in \cZono{\rM}$ define $\LL_{1,p} := p + \langle \rL_{1}\rangle$ to be the line through $p$ parallel to $\rL_{1}$ and let $q_{1} = \LL_{1,p} \cap H_{1}$ (see Figure \ref{fig:proofPix1}).

			\begin{figure}[htbp]
			\centering	
		        \includegraphics[width=.4\linewidth]{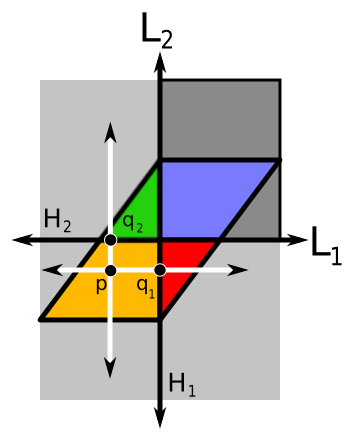}
		        \caption{A point $p \in \cZono{\rM}\cap \sigma_{(-,-)}$, the hyperplanes $H_{i}$ (black), and the lines           $\mathcal{L}_{i,p}$ (white)}
		        \label{fig:proofPix1}
		        \end{figure}
		  
				Since the width of $\cZono{\rM}$ parallel to $\rL_{i}$ is at most $||\rL_{i}||$, it follows that $p = \sum \alpha_{i}\rL_{i}$ for some unique set of $\alpha_{i}$ with $|\alpha_{i}| \le 1$.
				For example, when $i = 1$ we have $||p-q_{1}|| \le ||\rL_{1}||$ and,
				as $p-q_{1}$ is parallel to $\rL_{1}$ by construction, it follows that $p-q_{1} = \alpha_{1}\rL_{1}$ where 
				\[\alpha_{1} := \pm\frac{||p-q_{1}||}{||\rL_{1}||}\]
				is positive (respectively negative, zero) if and only if $p \in H^{+}_{1}$ (respectively $p \in H_{1}^{-}$, $p \in H_{1}$).
				
				Given $p = \sum \alpha_{i}\rL_{i}$, define the sign vector $\epsilon$ by 
					\[\epsilon_{i} = \begin{cases}
						\sign(\alpha_{i}) &\text{if $\alpha_{i} \neq 0$}\\
						+	&\text{else}.
					\end{cases}\]
					Then each $\delta_k$ in the expression
					\[p + v_{\epsilon} 
		                        = \sum_{i \in [d]} \alpha_{i}\rL_{i} 
		                        + \sum_{j:\,\epsilon_{j} = -}\rL_{j}
		                        = \sum_k \delta_k \rL_k
		                        \]
					is in $[0,1]$ and it follows that $Q_{\epsilon} \subseteq \rZ(\rL)$ (see Figure \ref{fig:proofPix2}).

			\begin{figure}[htbp]
				  \centering
			  \includegraphics[width=.4\linewidth]{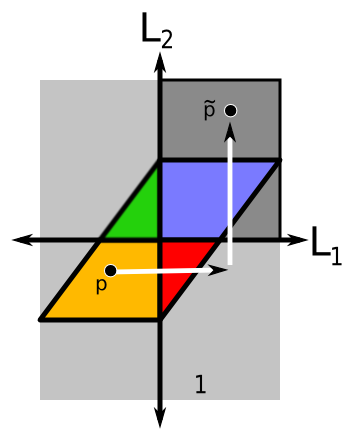}
				  \caption{The point $p$ and its shift $\tilde{p} = p + \rL_{1} + \rL_{2}$ into $\rZ(\rL)$.}
				  \label{fig:proofPix2}
			\end{figure}

				Now we prove $\rZ(\rL) \subseteq \bigcup_{\epsilon}Q_{\epsilon}$.
				Let $q = \sum_{i \in [d]} \gamma_{i}\rL_{i} \in \rZ(\rL)$, so that $\gamma_{i} \in [0,1]$ for all $i$ by definition.
				 Since facet-to-facet shifts of $\cZono{\rM}$ tile the column space of $\rM$, the point~$q$ lies in some translate of~$\cZono{\rM}$. 
				 Since to pass from one tile to a neighboring one through a facet is to add some vector $w$ in ${}_{\Z}\langle B \rangle = {}_{\Z}\langle \rL \rangle$ (Corollary~\ref{corollary:taleOfTwoLattices}), we have  $q \in \cZono{\rM} + \sum_{i \in [d]} a_{i}\rL_{i}$,
				where the $a_{i} \in \Z$, so that
					\[q = \sum_{i\in[d]} \alpha_{i}\rL_{i} + \sum_{i\in[d]} a_{i}\rL_{i}\]
					with $\alpha_{i} \in [-1,1]$. 
		Moreover, all $a_{i}\ge 0$ because $q$~lies in the positive hull of the $L_i$'s.
				Comparing coefficients in the two expressions for $q$ and using that the $\rL_i$ form a basis yields $\alpha_{i} + a_{i} = \gamma_{i}.$ 
				Since $a_{i}$ is a nonnegative integer and $\gamma_i\in [0,1]$, we have $a_{i} \in \{0,1\}$ (notice that the degenerate case $\alpha_i=-1$ and $\gamma_i=1$, in which case $a_i$~would equal~$2$, cannot occur), and 
					 \[a_{i}= 	\begin{cases}
					 				1 \text { if } \alpha_{i} \in [-1,0) \text{ and }\\
					 				0 \text{ if } \alpha_{i} \in (0,1].
					 			\end{cases} \]
				Let $\epsilon$ be the sign vector defined by $\epsilon_{i} = -$ (respectively $+$) if $a_{i} = 1$ (respectively,~$0$).
				Then $q \in Q_{\epsilon}$ and hence $\rZ(\rL) \subseteq \bigcup_{\epsilon} Q_{\epsilon}$.

				Finally, we show that for any two sign vectors $\epsilon, \epsilon'$ the intersection of the relative interiors of $Q_{\epsilon}$ and $Q_{\epsilon'}$ is empty. 
				Let $\phi :\bigcup(\Relint P_{\epsilon}) \to \bigcup(\Relint Q_{\epsilon})$ be the map that sends $\Relint P_{\epsilon} \to \Relint Q_{\epsilon}$.
				There are two points $p \neq p' \in \cZono{\rM}$ with $\phi(p) = \phi(p')=:q$ if and only if
					\[q \in (\cZono{\rM} + v_{\epsilon})\cap (\cZono{\rM} + v_{\epsilon'})\]
					for two sign vectors $\epsilon$ and $\epsilon'$.
					So $q$ lies on the boundary of both translates of $\cZono{\rM}$.
					But then $p$ and $p'$ both lie on the boundary of $\cZono{\rM}$ which contradicts the fact that they were in the relative interior of their respective cells. 
				Thus $\phi$ is a bijective map onto 
					\[\rZ(\rL) \setminus \left(\partial\rZ(\rL) \cup \bigcup_{\epsilon}\partial Q_{\epsilon}\right).\]
				So we have produced a volume-preserving bijection between $\rZ(\rL)$ and $\rZ(\rM)$ (up to a set of measure zero), which completes the proof.
			\end{proof}

			Note that all our proofs in this section go through in the case that the regular matroid~$\MM$ has loops or parallel elements by taking a unimodular representation $\rM$ of $\MM$ where the columns corresponding to loops are columns consisting only of zeros and the columns corresponding to parallel elements are all equal.
			Moreover, the zonotope $\rZ(\rM)$ is equal to the zonotope generated by the matrix $\rM'$ whose columns are the distinct nonzero columns of $\rM$ scaled by their multiplicity.
			This shows that, after an appropriate modification to the definition of $\rL$, Theorem \ref{theorem:polytopalMatroidMTT} is valid even after scaling the columns of the unimodular matrix $\rM$ by integers.
			\begin{cor}
			\label{cor:scaledPMMTT}
				Let $\MM$ be a regular matroid on $n$ elements represented by the unimodular matrix $\rM$ of full row rank, $D$ be a a $n \times n$ diagonal matrix with integer entries, and $\rM'=\rM D$.
				Then the volume of the zonotope $\rZ(\rM')$ equals the volume of $\rZ(\rL')$, where $\rL' = \rM D \rM^{\top}$.
			\end{cor}
			This result generalizes to the case that $D$ is a diagonal matrix with real entries, as we show in the next section.



		\section{Weighted Regular Matroids}
		\label{section:weightedCase}

			Let $\MM$ be a regular matroid on $n$ elements, $\rM$ be a unimodular representation of $\MM$, $D$ be a diagonal matrix with diagonal $\omega = (w_{1}, \dots, w_{n}) \in \R^{n}$, and $\rM_{\omega} = \rM D$ and $ \rL_{\omega} = \rM D \rM^{\top}$ be as in Subsection \ref{subsection:Matroids}.
			As scaling columns of $\rM$ does not affect the matroid $\MM$, we have $\MM(\rM) \isomorphic \MM(\rM_{\omega})$.
			In particular, scaling the columns of $\rM$ does not affect the cocircuits, and so ${}_{\Z}\left \langle \mathcal{C}^{*} \right \rangle = {}_{\Z}\left \langle \rM^{\top} \right \rangle$.

			Let $F$ be a facet of $\rZ_{0}(\rM_{w})$ corresponding to the cocircuit $C^{\star}$.
			Then $F$ is given by
			\[
				F_{C^{*}} 
				= \sum_{i: C^{*}_{i} = 1} \tfrac{1}{2}w_{i}\rM_{i} 
				- \sum_{i: C^{*}_{i} =-1} \tfrac{1}{2}w_{i}\rM_{i} 
				+ \sum_{i: C^{*}_{i} = 0} w_{i}S_{i},
			\]
			from which it is clear that the barycenter of $F_{C^{*}}$ is $\frac{1}{2}M_{\omega}C^{*}$.
			It follows that the lattice spanned by $\rL_{\omega}$ equals the lattice spanned by $\rM_{\omega}\mathcal{C}^{*}$, generalizing Corollary \ref{corollary:taleOfTwoLattices}.
			Replacing $\rM$ and $\rL$ in the proof of Theorem \ref{theorem:polytopalMatroidMTT} by $\rM_{\omega}$ and $\rL_{\omega}$, respectively, proves the following version of the matrix tree theorem for weighted regular matroids.			
			\begin{thm}
			 \label{thm:weightedPolytopalMatroidMTT}
			 	Let $\MM$ be a regular matroid on $n$ elements with full-rank unimodular representation $\rM$ and let $D = \diag(\omega)$ be an $n \times n$ diagonal matrix with real entries.
			 	Then $\vol(\rZ(\rL_{\omega})) = \vol(\rZ(\rM_{\omega}))$.
			\end{thm} 
			This result gives a new proof for Theorem 5.5 in \cite{an2014canonical} while simultaneously generalizing it from weighted graphs to regular matroids.
			Moreover, by Theorem \ref{theorem:associatedVolumes} and the fact that duals of regular matroids are regular, our result implies the dual version of the matrix tree theorem (see Theorem 5.2 in \cite{an2014canonical}), generalized to regular matroids.
			All of this is done without use of the Cauchy-Binet Theorem nor divisor theory on graphs.



		\section{The Graphical Case}
		\label{section:graphicCase}

			Let $G=([n],E)$ be a connected graph on $n$ vertices with signed vertex-edge incidence matrix $\rN$ and Laplacian $\rL$.	
			The rank of $\rN$ (and hence of $\rL$) is equal to the maximal size of a linearly independent subset of the columns of $\rN$.
			This is exactly the number of edges in a spanning tree of $G$, i.e., $\rank \rN = \rank \rL = n-1$.
			It follows that $0$ is an eigenvalue of $\rL$ of multiplicity $1$, and it is easy to check that the all-ones vector $\vecone_{n}$ is a corresponding eigenvector.
			So the zonotope $\rZ(\rL)$ is no longer a parallelepiped and its volume is no longer obtained by computing the determinant of $\rL$, as was the case in the previous section.
			Nonetheless, we now modify our techniques from the previous section to obtain a polyhedral proof of the classical matrix tree theorem.

			Recall from the introduction that the original formulation for the matrix tree theorem states that, for $G$ and $\rL = \rN\rN^{\top}$ as in the previous paragraph, the number of vertices times the number of spanning trees is equal to the product of the nonzero eigenvalues of $\rL$.
			The classical proof of this version of the matrix tree theorem proceeds in three steps.  
			First one uses the fact that $0$ is an eigenvalue of $\rL$ of multiplicity 1 with corresponding eigenvector $\vecone_{n}$ to show that all $n$ of the maximal principal minors of $\rL$ are equal and that the coefficient $c_{1}$ on the linear term of the characteristic polynomial of $\rL$ is equal to $n$ times any maximal principal minor.
			Then one uses the Cauchy-Binet theorem and the total unimodularity of $\rN$ to prove that each of these minors equals the number of spanning trees of $G$.
			Finally the theorem follows from the observation that, since $\rL$ is symmetric and $0$ is an eigenvalue of multiplicity $1$, the characteristic polynomial of $\rL$ factors over $\R$ and hence the coefficient $c_{1}$ is the product of the nonzero eigenvalues of $\rL$.
			Our polyhedral proof of the matrix tree theorem follows a similar tack. 
			
			First we show in Proposition \ref{proposition:ZLDecomposition} that the zonotope $\rZ(\rL)$ decomposes into $n$ parallelepipeds all having the same volume.
			Then we explain how results from the previous sections show that the volume of one (and hence any) of these parallelepipeds is equal to the number of spanning trees of $G$.
			Finally we show that the volume of $\rZ(\rL)$ is the product of the nonzero eigenvalues of $\rL$ as follows: First we construct two full-dimensional zonotopes, one having $d$-dimensional volume equal to $n$ times the $(d-1)$-dimensional volume of $\rZ(\rL)$ and the other having volume equal to $n$ times the product of the nonzero eigenvalues of $\rL$. 
			Then we show that these two zonotopes have the same volume using a proof technique reminiscent of that used to prove Theorem \ref{theorem:polytopalMatroidMTT}.  
			Moreover, we prove these results in greater generality whenever possible.

			Our first goal is to see how the factor of $n$ in the Matrix Tree Theorem manifests itself in the polyhedral set-up, the idea being that the zonotope of the Laplacian of $G$ is the union of $n$ zonotopes all having the same volume. 
			We formalize this in the following result which holds in the more general case that the matrix $\rM$ is only \defn{unimodular}, i.e., it has all \emph{maximal} minors in $\{-1,0,1\}$.
			\begin{prop}
			\label{proposition:ZLDecomposition}
				Let $\rM$ be a unimodular matrix and let $\rL = \rM\rM^{\top}$.
				Then the zonotope $\rZ(\rL)$ decomposes into $|\BB(\rM^{\top})|$ top dimensional parallelepipeds all having the same volume.
			\end{prop}
			\begin{proof}
				As $\im (\rM^{\top})$ is orthogonal to $\ker\rM$, an independent set in the matroid $\MM(\rM^{\top})$ remains independent after multiplication by $\rM$, i.e., $\MM(\rM^{\top})$ and $\MM(\rL)$ are isomorphic matroids.
				As $\rM$ is unimodular, so is $\rM^{\top}$, and so any set of columns $B$ of $\rM^{\top}$ corresponding to a basis of its matroid is a $\Z$-basis for the lattice $\LL = \Z^{n} \cap \im(\rM^{\top})$, that is, $\rZ(B)$ is a fundamental parallelepiped of $\LL$. 
				It follows that every top dimensional parallelepiped in a maximal cubical decomposition of $\rZ(\rL)$ is a fundamental parallelepiped for the lattice $\rM \LL = {_{\Z}}{\left \langle \rL  \right \rangle}$, the image of $\LL$ under $\rM$.
				The result now follows from the fact that the volume of a fundamental parallelepiped of a lattice is a lattice invariant.
			\end{proof}

		        \begin{figure}[htbp]
			\centering	
				\begin{subfigure}{.3\textwidth}
				  \centering
				  \includegraphics[width=\linewidth]{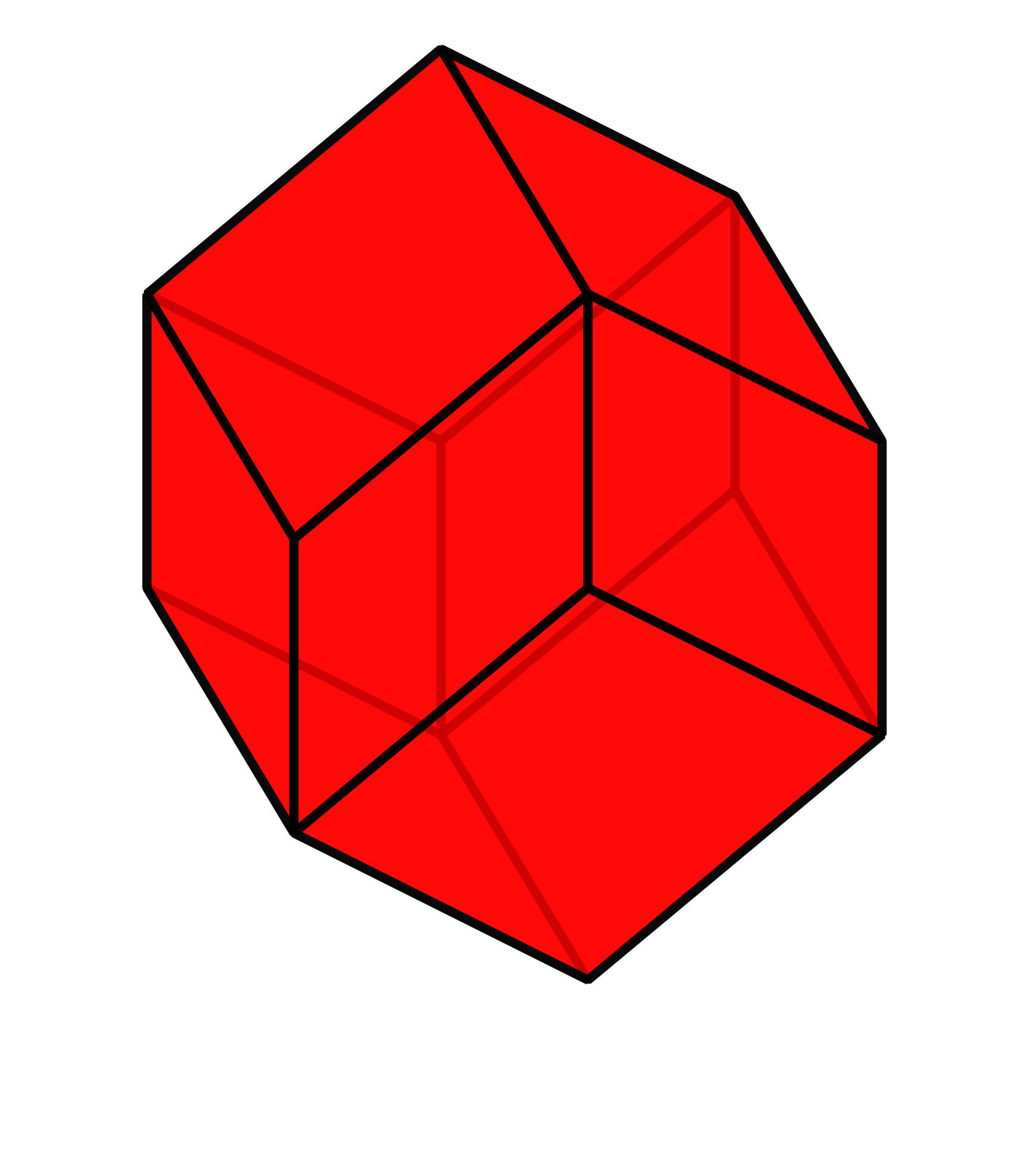}
				  \caption{$\rZ(\rL) \subset \R^{4}$ of $K_{4}$}
				  \label{fig:zlDecomp1}
				\end{subfigure}%
				\begin{subfigure}{.3\textwidth}
				  \centering
				  \includegraphics[width=\linewidth]{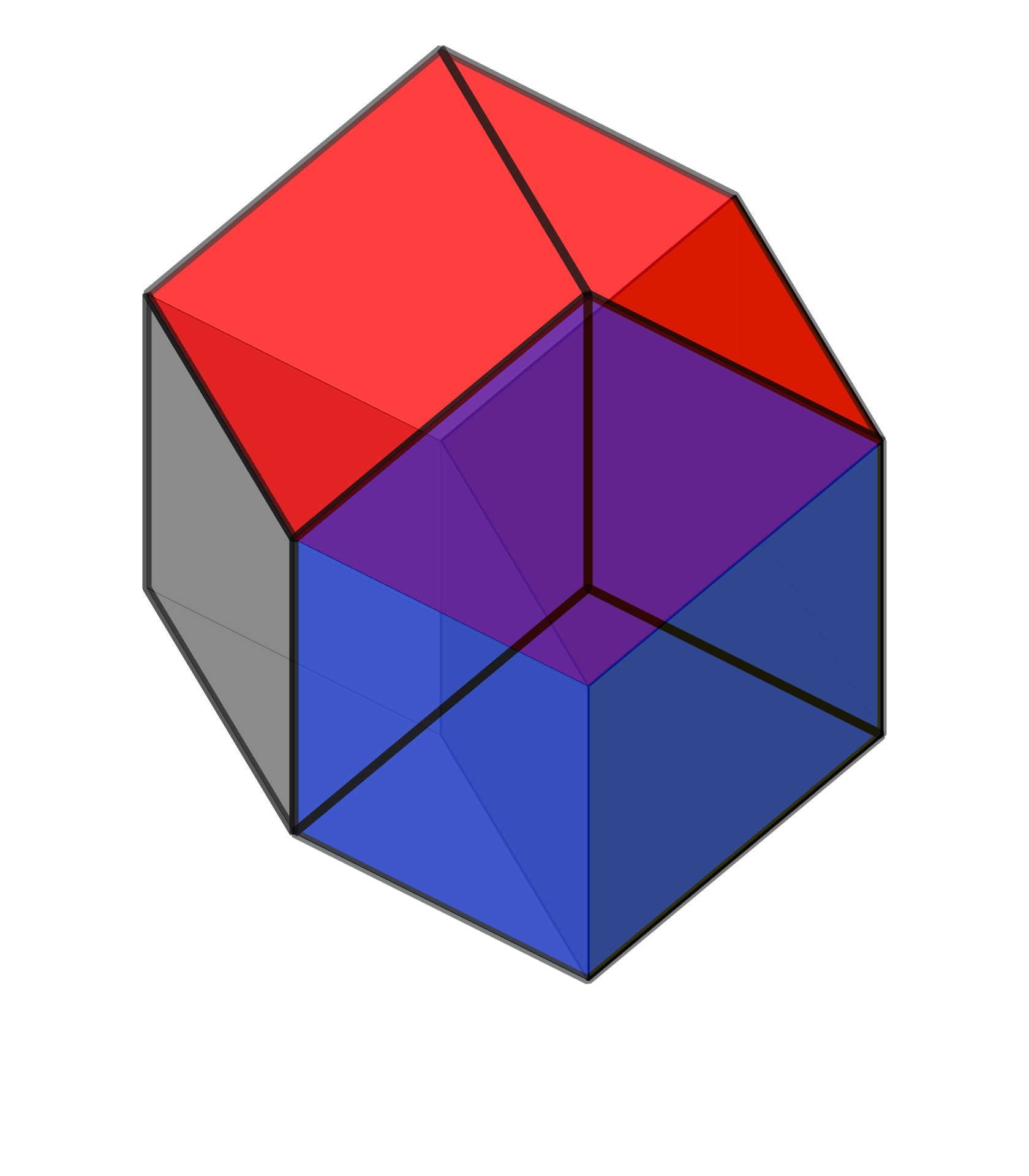}
				  \caption{Parallelepipeds of $\rZ(\rL)$}
				  \label{fig:zlDecomp2}
				\end{subfigure}
				\begin{subfigure}{.3\textwidth}
				  \centering
				  \includegraphics[width=\linewidth]{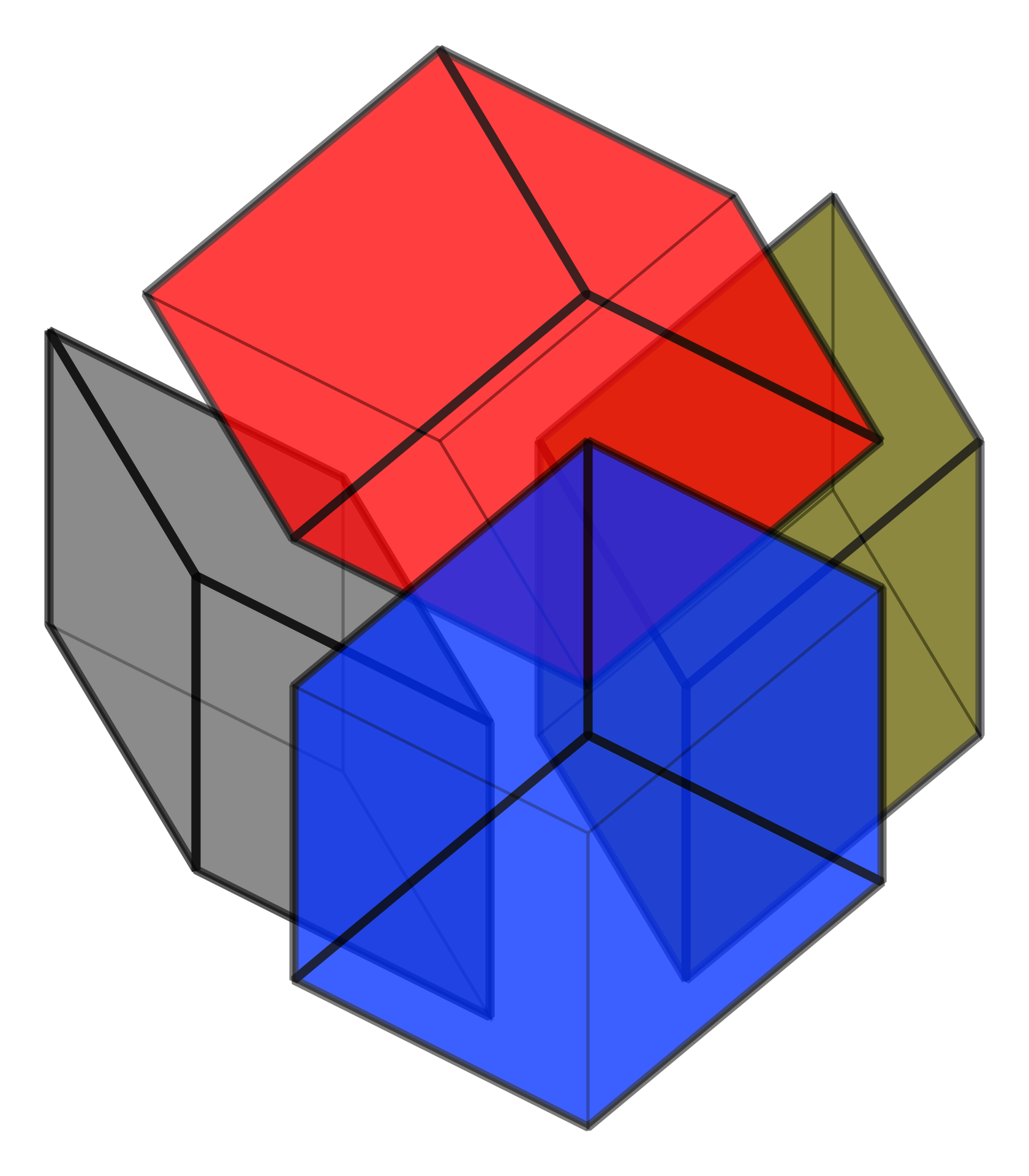}
				  \caption{An exploded view}
				  \label{fig:zlDecomp3}
				\end{subfigure}
			\caption{Proposition \ref{proposition:ZLDecomposition} at work on $\rZ(\rL)$ of the complete graph $K_{4}$.}
			\label{fig:ZLdecomp}
			\end{figure}

			\begin{example2*}
				Consider the complete graph $K_{4}$ on four vertices with edges oriented so that $i \rightarrow j$ if $i<j$. The signed vertex-edge incidence matrix $\rN$ and the Laplacian $\rL$ are 
			\[\rN = \begin{pmatrix}
				-1 & -1 & -1 &  0 &  0 & 0 \\
				 1 &  0 &  0 & -1 & -1 & 0 \\
				 0 &  1 &  0 &  1 &  0 & -1 \\
				 0 &  0 &  1 &  0 &  1 & 1
			\end{pmatrix},
			\hspace{44pt}
			\rL = \begin{pmatrix}
				3 & -1 & -1 & -1 \\
				-1 & 3 & -1 & -1\\
				-1 & -1 & 3	& -1\\
				-1 & -1 & -1 & 3	
			\end{pmatrix}.\]
			The three dimensional zonotope $\rZ(\rN) \subset \R^{4}$ is a translate of the classical permutahedron obtained by taking the convex hull of all points obtained from $[1,2,3,4]$ by permuting coordinates.

			The zonotope $\rZ(\rL)$ is the cubical zonotope (all of its facets are 2-cubes) displayed in Figure \ref{fig:zlDecomp1}. By Proposition \ref{proposition:ZLDecomposition} it is the union of four parallelepipeds of equal volume; see Figure \ref{fig:zlDecomp2} for the subdivision of $\rZ(\rL)$ into parallelepipeds and Figure \ref{fig:zlDecomp3} for an exploded view of the subdivision. 
			\end{example2*}

			In the graphical case, Proposition \ref{proposition:ZLDecomposition} tells us that the zonotope $\rZ(\rL)$ decomposes into $n$ parallelepipeds in ${_{\R}}{\left \langle \rN \right \rangle}$ all having the same volume.	
			More explicitly the decomposition is $\rZ(\rL) = \bigcup_{i} \Pi_{i}$ where, for $i \in [n]$, the parallelepiped $\Pi_{i}$ is generated by all of the columns of $\rL$ save for the $i^{\text{th}}$.
			We now show that the volume of one (and hence, any) of these parallelepipeds is equal to the number of spanning trees of $G$.
			To see this first note that Theorem \ref{corollary:volumeOfTUzonotope} holds regardless of the corank of the unimodular matrix involved, so in our case the volume of $\rZ(\rN)$ equals the number of spanning trees of $G$ (recall here that volume is taken with respect to the affine hull of the columns of $\rN$.)
			Also independent of the corank of the defining matrix is Lemma \ref{corollary:taleOfTwoLattices}, in which we showed that the lattice generated by the columns of the Laplacian is equal to the lattice generated by the matrix~$B$ whose columns are the barycenters of the facets of $\rZ(\rN)$ scaled by a factor of 2.
			Since any $n-1$ columns of $\rL$ form a lattice basis for 
			${_{\Z}}{\left \langle \rL \right \rangle} = {_{\Z}}{\left \langle B \right \rangle}$, we only need to check that an appropriate modification of Theorem \ref{theorem:polytopalMatroidMTT} still holds when we drop the full-rank condition.
			Indeed, in the proof of the theorem the full-rank condition guaranteed us that the columns of $\rL$ formed a basis for their $\Z$-span, whereas when the corank of $\rL$ is greater than 0 the columns over-determine the $\Z$-span.
			Nonetheless, the proof of Theorem \ref{theorem:polytopalMatroidMTT} at the end of Section~\ref{section:matroidProof} goes through verbatim for the following theorem in which $\rM$ is allowed to have arbitrary corank.
			\begin{thm}
			\label{thm:polytopalMMTTcorank}
				Let $\MM$ be a regular matroid and $\rM$ be a unimodular representation of $\MM$ over $\R$.
				Let $\rL = \rM\rM^{\top}$ and let $\overline{\rL}$  be the matrix obtained by taking any basis for $_{\Z}{\left \langle \rL \right \rangle}$ from among the columns of $\rL$.
				Then the volume of $\rZ(\rM)$ equals the volume of $\rZ\left(\overline{\rL}\right)$.
			\end{thm}
			\begin{proof}
			\end{proof}
			
			In the graphical case, taking Proposition \ref{proposition:ZLDecomposition} and Theorem \ref{thm:polytopalMMTTcorank} together shows that the volume of $\rZ(\rL)$ is $n$ times the number of spanning trees.
			So all that remains is to show that the volume of $\rZ(\rL)$ is the product of nonzero eigenvalues of $\rL$.
			We will achieve this by defining two new full-dimensional zonotopes $\rZ(\Lambda)$ and $\rZ(\Gamma)$ and then showing that 
		\begin{enumerate}[(i)]
		\item  $\vol\rZ(\Lambda)=n\lambda_1\cdots\lambda_{n-1}$, 
		\item $\vol\rZ(\Gamma)=n\vol\rZ(\rL)$, and 
		\item $\vol\rZ(\Lambda) = \vol\rZ(\Gamma)$.
		\end{enumerate}

		To construct these new zonotopes, define the matrices $\Lambda$ and $\Gamma$ by setting $\Lambda_{ij} = \rL_{ij}+1$ and letting $\Gamma=[\rL|\vecone]$ be the matrix obtained from~$\rL$ by appending a column of ones.

		To prove (i), observe that the columns of $\Lambda$ arise by summing the vector $\vecone$ to each column of the rank $(n-1)$ matrix~$\rL$, and that $\vecone$~is orthogonal to each of these columns. 
		In consequence, the columns of the $n \times n$ matrix~$\Lambda$ are linearly independent. 
		Thus, the zonotope $\rZ(\Lambda)$ is an $n$-dimensional parallelopiped with volume equal to the product of the eigenvalues of $\Lambda$.
		If $\lambda \in \Spec(\rL)$ is a nonzero eigenvalue with eigenvector $v$, then the sum of the coordinates of $v$ is zero.
		It follows that $\Lambda v = \rL v = \lambda v$, and so $\lambda$ is also an eigenvalue of~$\Lambda$.
		Since $\vecone \in \ker L$, it follows that $\Lambda \vecone = n \vecone$, and so $\Spec\Lambda = (\Spec(\rL) \setminus \{0\}) \cup \{n\}$, and $\vol\rZ(\Lambda) = n\,\lambda_{1}\cdots \lambda_{n-1}$.

		For (ii), first observe that $\det(\rN_{P_n}|\vecone)=n$, where $\rN_{P_n}$ ~is the signed incidence matrix of the path on $n$~vertices. Thus, the volume of any zonotope that is a prism $\rZ(M|\vecone)=\rZ(M)\times\vecone$ over a unimodular cube~$\rZ(M)$ is~$n$.
		Our claim $\vol\rZ(\Gamma) = n \vol \rZ(\rL)$ now follows from the following general fact:
		\begin{prop}
		\label{proposition:prism}
		Let $P \in \R^{n}$ be an $(n-1)$-dimensional lattice polytope with affine span~$S$ and let $\mathcal{L} = S \cap \Z^{n}$ be the induced lattice.
		For $v \in \Z^{n}\setminus S$ let $Q$ be the prism $P \times v$.
		Then 
		\[
		   \vol (Q) 
		   \ = \
		   h_{S}(v)\vol_{S}(P),
		\] 
		where $h_{S}(v)$ is the lattice height of $v$ from $S$ and $\vol_{S}$ is the induced volume form on~$\aff S$.
		\end{prop}

		\begin{proof}

		Without loss of generality we may assume that $\veczero \in S$ so that $S$ is a linear hyperplane with primitive normal vector $u\in\Z^d$, say.
		For any $i \in \Z$ define $S_{i}$ to be the parallel translate of $S$ given by $\{x \in \R^{n} \suchthat \langle x,u\rangle = i\}$.
		Then for every $v \in \Z^{n}$ there is an $i \in \Z$ such that $v \in S_{i}$, and this is precisely the lattice height of $v$ with respect to $S$, $h_{S}(v) = i$.

		Suppose $v \in \Z^{n}$ satisfies $h_{S}(v)=1$.
		In the $k$-th dilate of $Q$, the only lattice points of~$\Z^n$ lie on the sections $Q \cap H_{i}$ where $H_i=\{x\in\R^n:h_S(x)=i\}$ for $0\le i\le k$.
		Moreover, the distribution of lattice points is the same in each section $Q\cap H_i$.
		Thus, the number of lattice points in the $k^{\text{th}}$ dilate of $Q = P \times v$ is exactly 
		\begin{align*}
		\#\left (Q \cap \frac{1}{k}\Z^{n}\right)	&=	(k\,h_{S}(v)+1)\;\#\!\left(P \cap \frac{1}{k}\mathcal{L}\right) \\
							&= (k+1)\;\#\!\left(P \cap \frac{1}{k}\mathcal{L}\right).
		\end{align*}
		So in this case we have 
			\begin{align*}
				\vol(Q)	&=	\lim_{k \to \infty} \frac{1}{k^{n}}\#\left(Q \cap \frac{1}{k}\Z^{n}\right)\\
						&=	\lim_{k \to \infty} \frac{k+1}{k}\frac{1}{k^{n-1}}\#\left(P \cap \frac{1}{k}\mathcal{L}\right)\\
						&=	\lim_{k \to \infty} \frac{1}{k^{n-1}}\#\left(P \cap \frac{1}{k}\mathcal{L}\right)\\
						&=	\vol_{S}(P).
			\end{align*}
		Since $Q$ is a full-dimensional prism, its lattice volume and Euclidean volume coincide.
		It follows that $\vol(P \times v) = \vol_{S}(P)$  for any $v \in \R^{n}$ with $h_{S}(v) =1$.

		For an arbitrary $v \in \Z^{n}$ with $h_{S}(v)=i$, the prism $Q$ decomposes into $i$ (typically rational) polytopes which are slices of $Q$ sitting between the affine hyperplanes $S_{j-1}$ and $S_{j}$ where $j \in [i]$.
		Each of these slices is a translated copy of a height-one prism over $P$ and hence has volume $\vol(P)$. 
		As there are $h_{S}(v)$ many of them, the result follows.
		\end{proof}

		The missing claim (iii), $\vol\rZ(\Lambda) = \vol\rZ(\Gamma)$, is true in much greater generality, and it is this generalization that we state in Theorem \ref{thm:thm4}, the proof of which uses a technique analogous to the proof of Theorem \ref{theorem:polytopalMatroidMTT}.

			\begin{thm}\label{thm:thm4}
				For any set $B = \{b_{1}, \dots, b_{n}\}$ of points that linearly span~$\R^n$,  let $\beta = \frac{1}{n}\sum_{i\in [n]}b_{i}$ be their barycenter and let $\Pi = \rZ(B)$ be the zonotope they generate.
				Let $P$ be the zonotope generated by $\beta$ together with the points $b_{i} - \beta$ for $i \in [n]$.
				Then $\vol \Pi = \vol P$. 
			\end{thm}

		Note that we obtain claim (iii) as a special case by taking $B, \Pi, $ and $P$ to be the columns of $\Lambda$, the zonotope $\rZ(\Lambda)$, and the zonotope $\rZ(\Gamma)$, respectively.
		Before proceeding with the proof in the general case, let us illustrate the techniques to be used:

		\begin{example2*}
		For the complete graph $K_{3}$ on three vertices, the zonotope $\rZ(\Gamma)$ is the prism over the hexagon $\rZ(\rL)$ shown in blue in Figure \ref{fig:thm4} intersecting the red parallelepiped $\rZ(\Lambda)$.

		\begin{figure}[htbp]
		  \centering
		        \includegraphics[width=.2\linewidth]{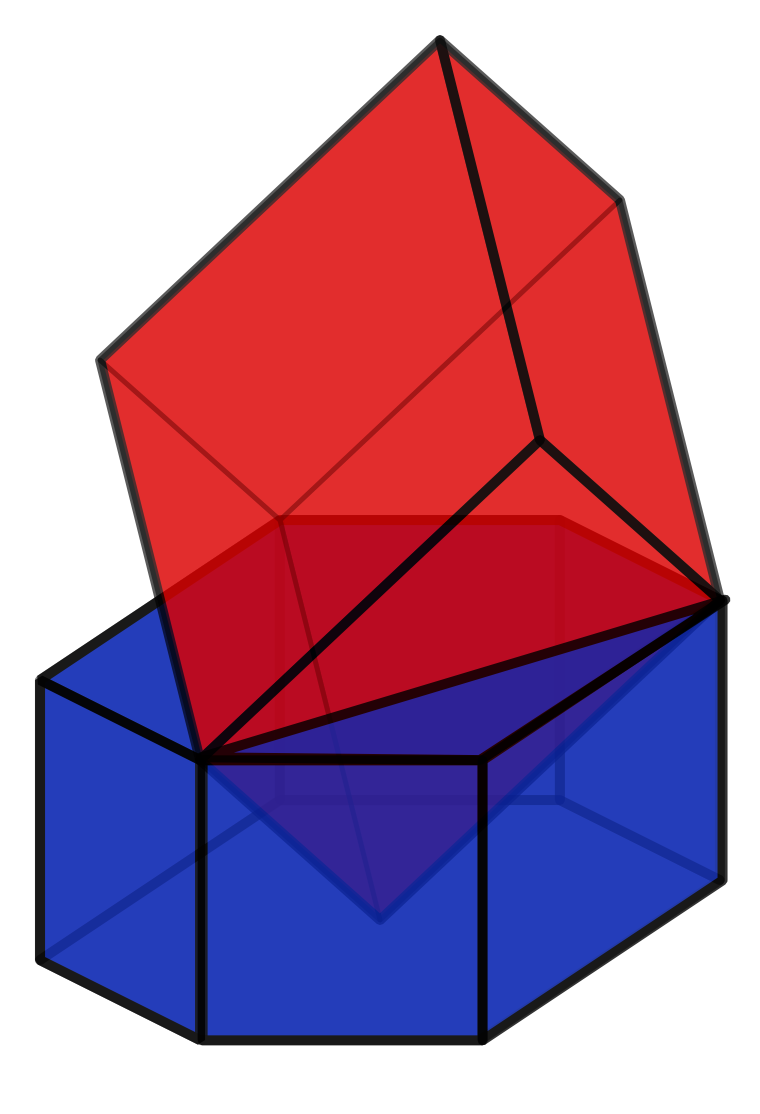}
		  \caption{The zonotopes $P = \rZ(\Gamma)$ and $\Pi = \rZ(\Lambda)$ appearing in Theorem \ref{thm:thm4} in the case of the graph $K_{3}$.}
		  \label{fig:thm4}
		\end{figure}

		For each sign vector $\epsilon \in \{+, -\}^{3}$, the simplicial cone spanned by $\epsilon \rL = \{\epsilon_{i}\rL_{i}\}$ intersects $\rZ(\Gamma)$ and these intersections are the~$P_{\epsilon}$.
		By construction, all of the $P_{\epsilon}$ are full-dimensional except for $P_{\{-,-,-\}}$ which consists only of the origin.
		The seven full-dimensional pieces are illustrated center-left in Figure~\ref{fig:thm4more}.

		        \begin{figure}[htbp]
			\centering	
		        \includegraphics[height=6.3cm]{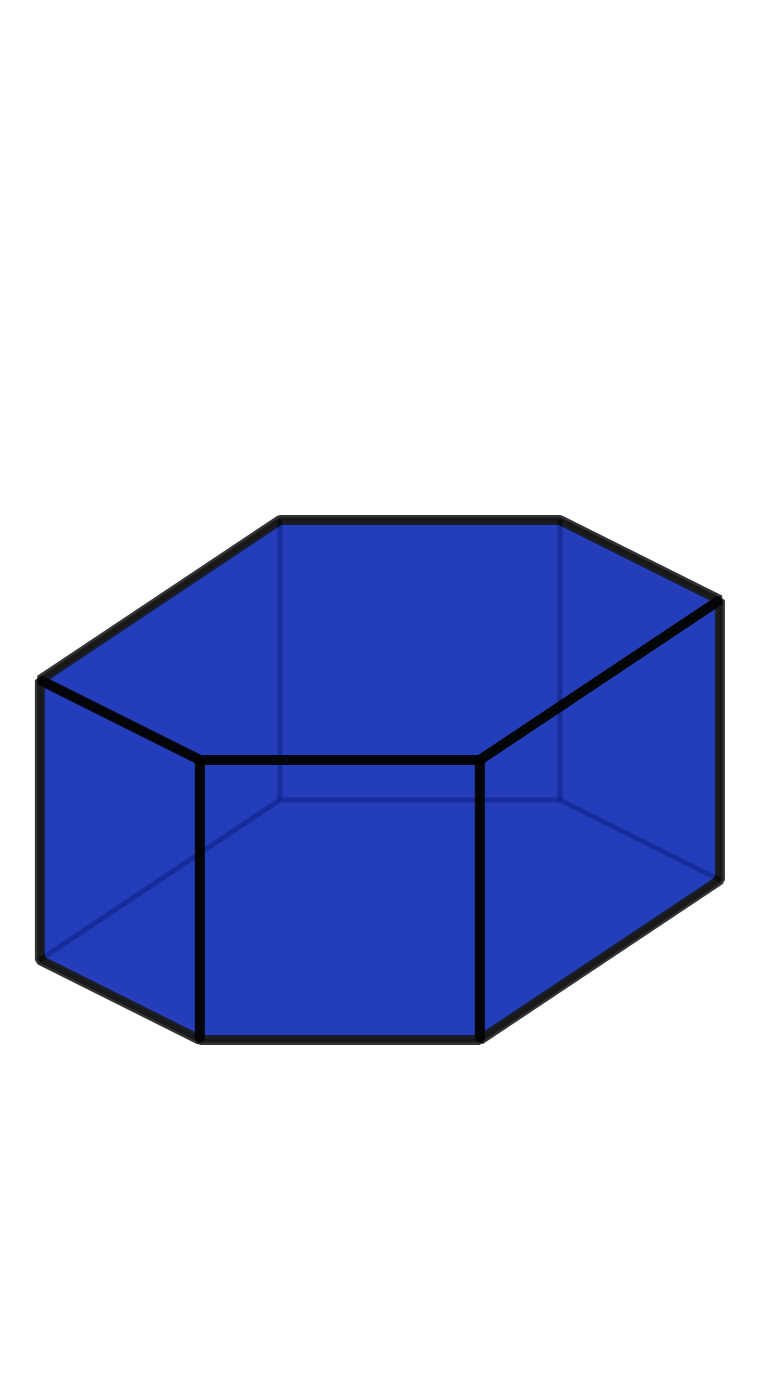}
			\hfill
		        \includegraphics[height=6.3cm]{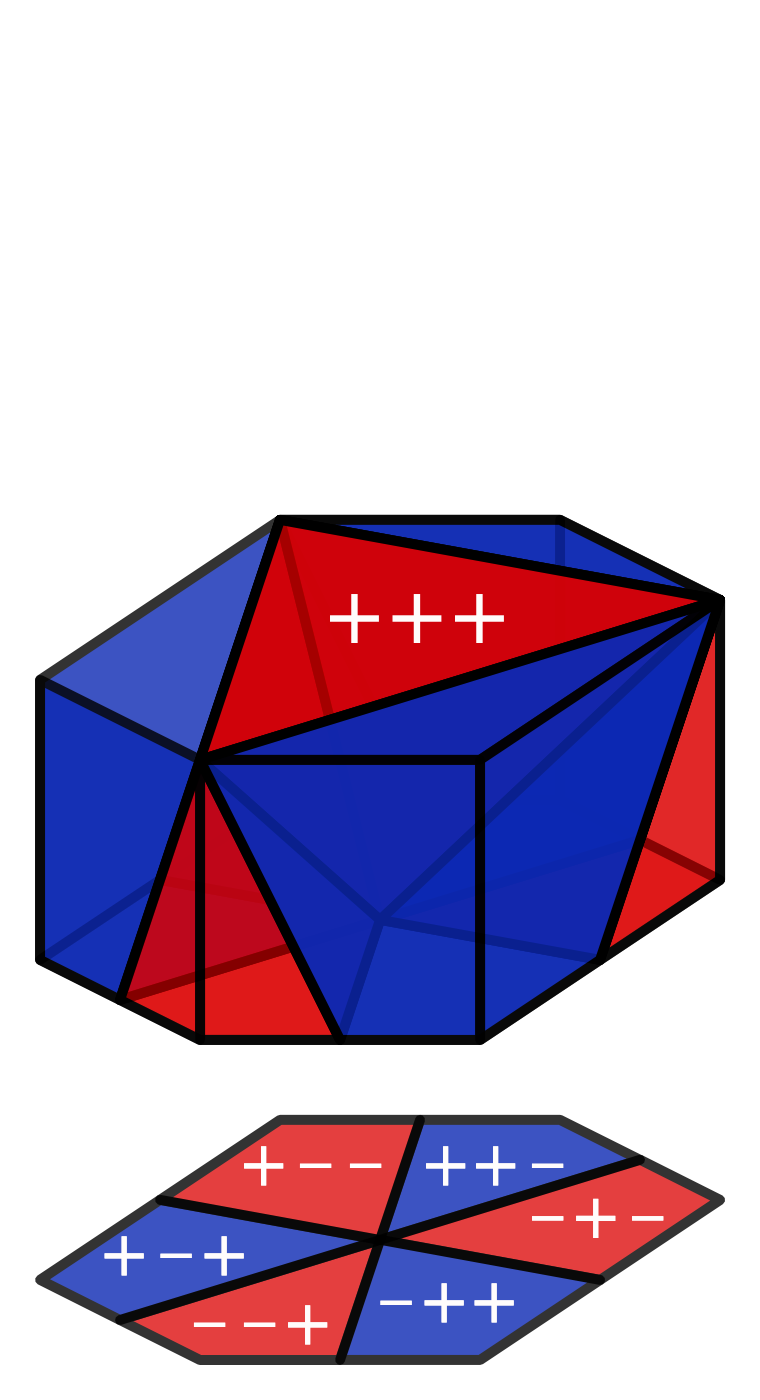}
		        \hfill
		        \includegraphics[height=6.3cm]{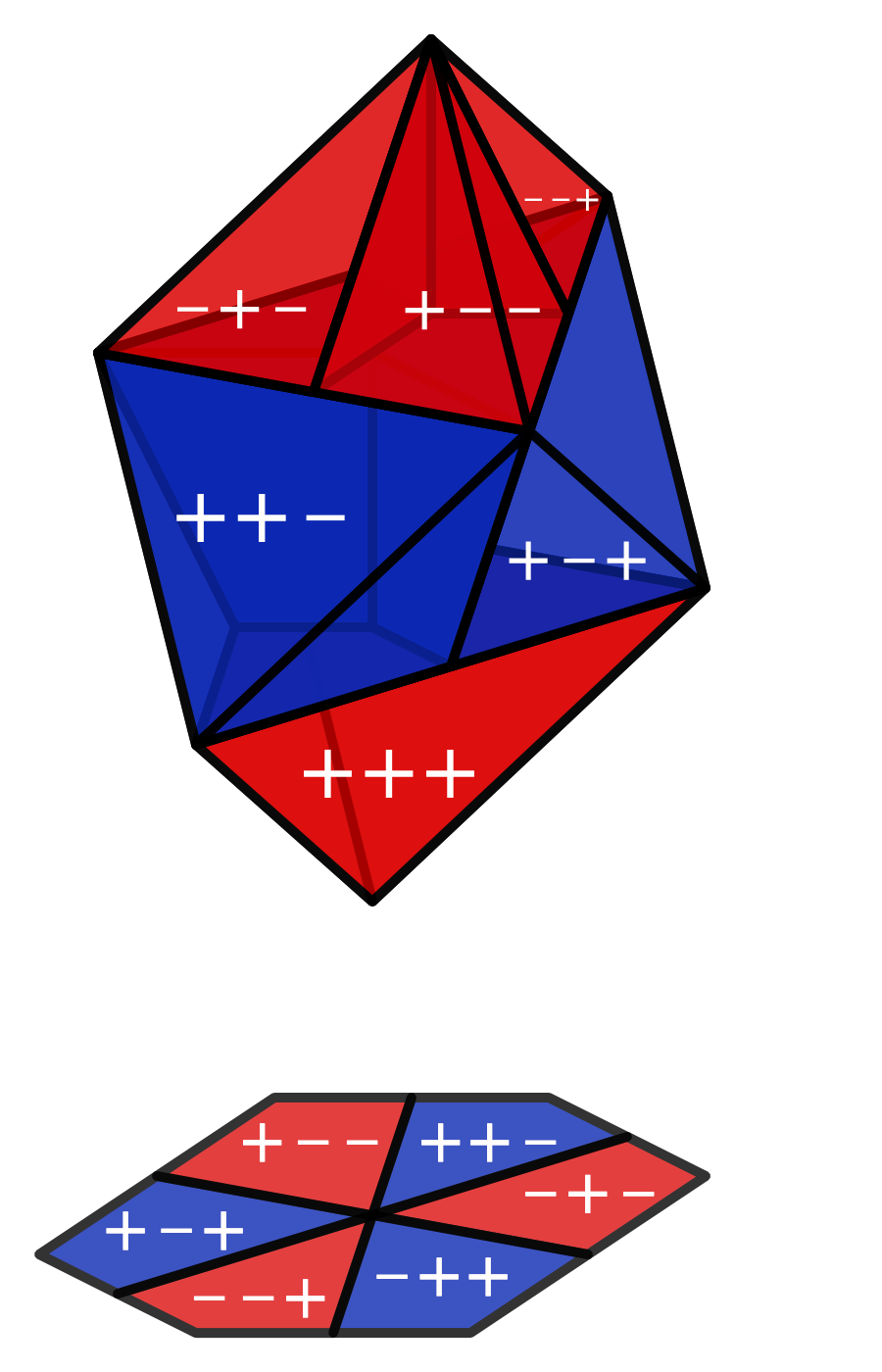}
		        \hfill
		        \includegraphics[height=6.3cm]{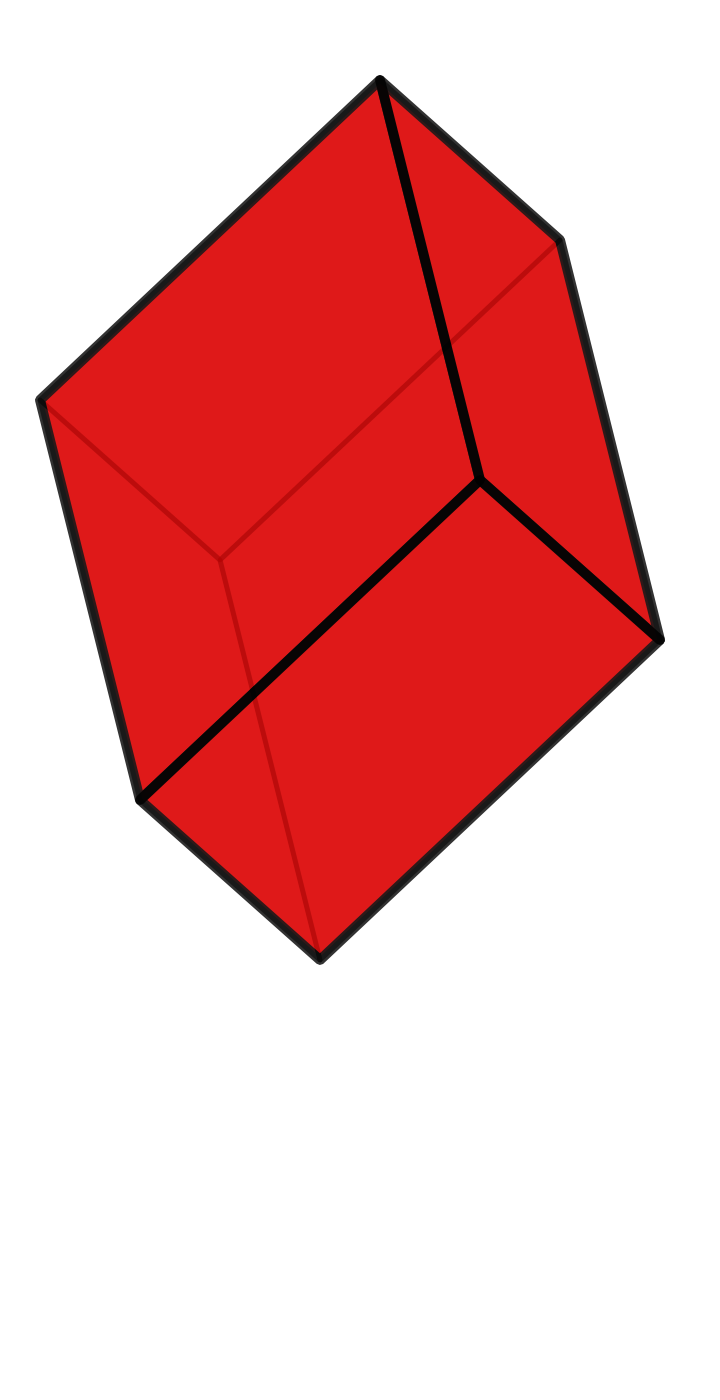}
		          \caption{The zonotopes of Theorem~\ref{thm:thm4} in the case of the graph $K_3$. From left to right: (i) $P=\rZ(\Gamma)$, (ii) the decomposition $Z(\Gamma)=\bigcup P_\epsilon$, (iii) the rearrangement $\bigcup Q_{\epsilon}$, and (iv) the parallelepiped $\rZ(\Lambda) = \bigcup Q_{\epsilon}$.}
		          \label{fig:thm4more}
		        \end{figure}

		Six of the seven $P_{\epsilon}$ are visible in the figure, while the colored hexagon beneath the prism suggests the location of the invisible piece.
		By translating each $P_{\epsilon}$ by the sum of all $\Lambda_{i}$ such that $\epsilon_{i}$ is negative, we obtain the union of the~$Q_{\epsilon}$ as seen center-right in Figure~\ref{fig:thm4more}.
		This union is exactly the zonotope of $\Lambda$.
		\end{example2*}

			\begin{proof}[Proof of Theorem~\ref{thm:thm4}]
				We prove that there is a decomposition of $P$ into full dimensional polytopal cells and a set of translations (one for each polytope in the decomposition) such that the union of the translated cells is exactly $\Pi$ and that if two shifted cells intersect, they do so only on their boundaries.
				
				First we show that for every point $p \in P$ there is a sign vector $\epsilon = \epsilon(p) \in \{+,-\}^{n}$ such that $ p \in \rZ(\epsilon B)$ where $\epsilon B := \{ \epsilon_{i}b_{i} \suchthat i \in [n]\}$.
				As $P$ is a zonotope, given any $p \in P$ there is an $\alpha \in [0,1]^{n+1}$ such that 
				 	\begin{align*}
						p	&=	\alpha_{n+1}\beta + \sum_{i \in [n]}\alpha_{i}(b_{i} - \beta)\\
							&=	\sum_{i\in[n]} \frac{1}{n}\left(n\alpha_{i} + \alpha_{n+1}- \sum_{j \in [n]}\alpha_{j}\right)b_{i}\\
							&= 	\sum_{i\in[n]} \frac{1}{n} \left( (n-1)\alpha_{i} + \alpha_{n+1}- \sum_{j \in [n] \setminus \{i\}}\alpha_{j}\right)b_{i}. 
					\end{align*}
				Let us abbreviate this last expression to $p=\sum_{i\in[n]}\gamma_i b_i$, where the $\gamma_i$ are unique because the $b_i$ form a basis of~$\R^n$. Since each $\alpha_{j}$ is in $[0,1]$, it follows that $\gamma_i\in[-1,1]$ for all~$i$. Therefore, setting $\epsilon_{i}=\sign\gamma_i$ if $\gamma_i\ne0$ (and $\epsilon_i=\pm$ arbitrarily if $\gamma_i=0$) proves the claim.

				For each $\epsilon \in \{+,-\}^{n}$, define $P_{\epsilon} := P \cap \rZ(\epsilon B)$ and $v_{\epsilon} = \sum_{i: \epsilon_{i} = -} b_{i}$, see Figure~\ref{fig:thm4}.
				By the previous paragraph we know that $P$ is the union of the $P_{\epsilon}$ and we now show that the union of the translated polytopes $P_{\epsilon}+ v_{\epsilon}$ is $\Pi$.
				To see this let $q = \sum_{i\in [n]} \alpha_{i}b_{i} \in \rZ(B)$.
				If $q = 0$ then $q \in P$ so we may assume there is a nonnegative integer $k$ such that $\sum_{i\in[n]} \alpha_{i} \in (k,k+1]$. 
				Moreover, we may assume (after permuting indices if necessary) that the $\alpha_{i}$ are decreasing, i.e., $\alpha_{1} \ge \alpha_{2} \ge \cdots \ge \alpha_{n}$.
				Now we define $\epsilon$ to be the sign vector with $\epsilon_{i} = -$ if and only if $i \le k$.
				It follows that
				\begin{align*}
					q - v_{\epsilon}	&=	q - \sum_{i = 1}^{k}b_{i}\\
								&=	
		  \sum_{i<k}(\alpha_{i} - 1)(b_{i}-\beta + \beta) 
		+ (\alpha_k - 1) b_k
		+ \sum_{j>k}\alpha_{j}(b_{j}-\beta+\beta)\\
								&=	\sum_{i < k}(\alpha_{i}-1)(b_{i}-\beta) + (\alpha_{k}-1)b_{k} + \sum_{j>k}\alpha_{j}(b_{j}-\beta) + \left(-(k-1) + \sum_{i\in[n]\setminus k}\alpha_{i}\right)\beta.
		\end{align*}
		Since $b_k = n\beta - \sum_{i\ne k} b_i$, we can express the second summand as
		\begin{align*}
		  (\alpha_k - 1)b_k 
		  &=
		  - (\alpha_k - 1) \left(-n\beta + \sum_{i\ne k}b_i\right) 
		  \\ &=
		  - (\alpha_k - 1) \left(-\beta + \sum_{i\ne k}(b_i-\beta) \right)
		  \\ &=
		  \sum_{i < k} (1-\alpha_k)(b_i - \beta) +
		  (\alpha_k-1)\beta +
		  \sum_{j > k} (1-\alpha_k)(b_j - \beta),
		\end{align*}
		so that 
		\begin{align*}
		q - v_{\epsilon}
								&=	\sum_{i<k} (\alpha_{i}-\alpha_{k})(b_{i}-\beta) + \sum_{j>k}(\alpha_{j}-\alpha_{k}+1)(b_{j} - \beta) + \left(-k + \sum_{i\in[n]}\alpha_{i}\right)\beta
				\end{align*}
				and so $q-v_{\epsilon} \in P$ since $\alpha_{i} \ge \alpha_{k}$ if $i \le k$ and $\alpha_{k}\ge \alpha_{i}$ otherwise. Moreover, our choice of~$k$ guarantees that all coefficients in this linear combination lie in $[0,1]$.
				
				 Finally, in order to prove that our decomposition and rearrangement preserves volume, we must show that if two translated cells $P_{\epsilon}+v_{\epsilon}$ and $ P_{\epsilon'} + v_{\epsilon'}$  intersect then they do so on a set of measure zero.
				 To see this let $p \in P_{\epsilon}$ and $p' \in P_{\epsilon'}$ be such that $\epsilon \neq \epsilon'$ and $p+v_{\epsilon} = p' + v_{\epsilon'} \in \rZ(B)$, where $v_{\epsilon} = \sum_{i: \epsilon_{i} = -} b_{i}$ as before.
				 Then 
				 \begin{align*}
				 	0 	&=	p + v_{\epsilon} - p' - v_{\epsilon'}\\
						&=	\sum \alpha_{i}b_{i} + v_{\epsilon} - \sum \beta_{i}b_{i} - v_{\epsilon'}\\
						&=	\sum_{i \in \epsilon^{-} \setminus \epsilon'^{-}}(\alpha_{i}-\beta_{i}+1)b_{i} + \sum_{j \in \epsilon'^{-} \setminus\epsilon^{-}} (\alpha_{j} - \beta_{j}-1)b_{j} + \sum_{k:\epsilon_k = \epsilon'_k} (\alpha_{k} -\beta_{k})b_{k}.
				 \end{align*}
				Since $B$ is a basis it follows that the coefficient on any $b_{i}$ in the final expression equals zero.
				Therefore $\alpha_{k} = \beta_{k}$ if $\epsilon_{k} = \epsilon'_{k}$ and otherwise either $\alpha_{i} = 0$ and $\beta_{i} = 1$ or vice versa.
				 It follows from the definitions of $P_{\epsilon}$ and $P_{\epsilon'}$ that $p \in \partial P_{\epsilon}$ and $p' \in \partial P_{\epsilon'}$, and the proof is complete.
			\end{proof}


		Taken together, Proposition \ref{proposition:prism} and Theorem \ref{thm:thm4} tell us that when $\rM$ is a unimodular matrix with corank 1, we can recover the product of the nonzero eigenvalues of $\rL$ by constructing a certain full-rank matrix $\Lambda$ associated to $\rL$ and analyzing the zonotope it generates.
		This construction essentially replaces the eigenvalue $0$ of $\rL$ with the eigenvalue~$n$ while fixing the other eigenvalues.
		We suspect that this can be strengthened to allow for unimodular representations of regular matroids of arbitrary corank in the statement of Theorem~\ref{theorem:polytopalMatroidMTT}.
		Presently we have no proof for this fact, and so we leave it as a conjecture.
		\begin{con}
		\label{conjecture:arbitraryCorank}
		 	Let $\MM$ be a regular matroid and $\rM$ a unimodular $m \times n$ representation of $\MM$ with corank greater than 1.
		 	Then there is a $m \times m$ matrix $\Lambda$ with full rank such that every nonzero eigenvalue of $\rL$ is an eigenvalue of~$\Lambda$ and every other eigenvalue of~$\Lambda$ depends only on the ambient dimension $m$.	
		 \end{con} 

\section*{Acknowledgements} 
                   We would like to extend our thanks to Raman Sanyal for stimulating discussions, and to Farbod Shokrieh for pointing out \cite{an2014canonical} to us.


\bibliographystyle{siam}
\bibliography{matrixTree3}

\begin{thebibliography}{10}

\bibitem{aigner2000proofs}
{\sc M.~Aigner and G.~M. Ziegler}, {\em Proofs from THE BOOK}, Springer-Verlag.
  Berlin, 2000.

\bibitem{an2014canonical}
{\sc Y.~An, M.~Baker, G.~Kuperberg, and F.~Shokrieh}, {\em Canonical
  representatives for divisor classes on tropical curves and the matrix-tree
  theorem}.
\newblock \url{http://arxiv.org/abs/1306.5351}.

\bibitem{beck2014enumerating}
{\sc M.~Beck, F.~Breuer, L.~Godkin, and J.~L. Martin}, {\em Enumerating
  colorings, tensions and flows in cell complexes}, Journal of Combinatorial
  Theory, Series A, 122 (2014), pp.~82--106.

\bibitem{bjorner1999oriented}
{\sc A.~Bj{\"o}rner, M.~Las~Vergnas, B.~Sturmfels, N.~White, and G.~M.
  Ziegler}, {\em Oriented matroids, volume 46 of {E}ncyclopedia of
  {M}athematics and its {A}pplications}, 1999.

\bibitem{brouwer2011spectra}
{\sc A.~E. Brouwer and W.~H. Haemers}, {\em Spectra of graphs}, Springer, 2011.

\bibitem{kirchhoff1847ueber}
{\sc G.~Kirchhoff}, {\em Ueber die {A}ufl{\"o}sung der {G}leichungen, auf
  welche man bei der {U}ntersuchung der linearen {V}ertheilung galvanischer
  {S}tr{\"o}me gef{\"u}hrt wird}, Annalen der Physik, 148 (1847), pp.~497--508.

\bibitem{maurer1976matrix}
{\sc S.~B. Maurer}, {\em Matrix generalizations of some theorems on trees,
  cycles and cocycles in graphs}, SIAM Journal on Applied Mathematics, 30
  (1976), pp.~143--148.

\bibitem{mcmullen1984volumes}
{\sc P.~McMullen}, {\em Volumes of projections of unit cubes}, Bulletin of the
  London Mathematical Society, 16 (1984), pp.~278--280.

\bibitem{oxley1992matroid}
{\sc J.~G. Oxley}, {\em Matroid theory}, vol.~3, Oxford {U}niversity {P}ress,
  1992.

\bibitem{schrijver1998theory}
{\sc A.~Schrijver}, {\em Theory of {L}inear and {I}nteger {P}rogramming}, John
  Wiley \& Sons, 1998.

\bibitem{shephard1974space}
{\sc G.~C. Shephard}, {\em Space-filling zonotopes}, Mathematika, 21 (1974),
  pp.~261--269.

\bibitem{stanley1991zonotope}
{\sc R.~P. Stanley}, {\em A zonotope associated with graphical degree
  sequences}, Applied Geometry and Discrete Mathematics: The Victor Klee
  Festschrift, 4 (1991), pp.~555--570.

\bibitem{tutte1971matroids}
{\sc W.~Tutte}, {\em Introduction to the theory of matroids}, American Elsevier
  Publishing Company, 1971.

\bibitem{vallentin2004note}
{\sc F.~Vallentin}, {\em A note on space tiling zonotopes},
  \url{http://arxiv.org/abs/math/0402053},  (2004).

\end{thebibliography}

\end{document}